\documentclass[graybox]{svmult}

\usepackage{mathptmx}       
\usepackage{helvet}         
\usepackage{courier}        
\usepackage{type1cm}        
\usepackage{makeidx}         
\usepackage{graphicx}        
\usepackage{multicol}        
\usepackage[bottom]{footmisc}

\makeindex             

\begin{document}

\titlerunning{Almost periodic evolution equations}

\title*{Almost Periodic Solutions of
Evolution Differential Equations with Impulsive Action }

\author{Viktor Tkachenko}

\institute{Viktor Tkachenko \at Institute of Mathematics National Academy of Sciences of Ukraine,
Tereshchenkivska str. 3, Kiev, Ukraine \email{vitk@imath.kiev.ua}}

\maketitle

\abstract*{In an abstract Banach space we study conditions for the existence of piecewise continuous, almost periodic solutions
for semilinear impulsive differential equation
with fixed and non-fixed moments of impulsive action.}

\abstract{In an abstract Banach space we study conditions for the existence of piecewise continuous, almost periodic solutions
for semilinear impulsive differential equation
with fixed and non-fixed moments of impulsive action.}

\section{Introduction}
\label{sec:1}
We consider the problem of the existence of piecewise continuous, almost periodic solutions for the nonlinear impulsive differential equation
\begin{eqnarray} \label{ban1}
& & \frac{du}{dt} + (A + A_1(t)u = f(t, u), \quad t \not= \tau_j(u), \\
& & u(\tau_j(u) + 0) - u(\tau_j(u)) = B_j u + g_j(u), \quad j \in  Z, \label{ban2}
\end{eqnarray}
where $u: {R} \to X,$ $X$ is a Banach space, $A$ is a sectorial operator in $X$, $A_1(t)$ is some operator-value function,
$\{B_j\}$ is a sequence of some closed operators, and $\{\tau_j(u)\}$ is an unbounded and strictly increasing sequence of real numbers for all $u$
from some domain of space $X.$

We use the concept of piecewise continuous almost periodic functions  proposed in \cite{HW}.
Points of discontinuities of these functions coincide to points of impulsive actions $\{\tau_j\}.$
We mention the remarkable paper \cite{SPA}, where a number of important statements about almost periodic pulse system was proved.
Then these results  were included in the well-known monograph \cite{SP}.
Today there are many articles related to the study of almost periodic impulsive systems ( see, for example,  \cite{A,AP,AP2,MT1,MT2,PR,STr,S,Tk}).
In the papers \cite{HAR,SA,T3,Tr} almost periodic solutions for abstract impulsive differential equations in the Banach
space are investigated.

In this paper, we consider the semilinear abstract impulsive differential equation in a Banach space with sectorial operator in the linear part of
the equation and some closed operators in linear parts of impulsive action.
Using fractional powers of operator $A$ and corresponding interpolation spaces allows us to consider strong or classical solutions.
Note that such equations with periodic right-hand sides were first studied in \cite{RT}.
In equations with  nonfixed moments of impulsive action, points of discontinuity depend on solutions; that is, every solution has
its own points of  discontinuity. Moreover, a solution can intersect the surface of impulsive action several times
 or even an infinite number of times. This is the so-called pulsation or beating phenomenon.
We will assume that solutions of (\ref{ban1}), (\ref{ban2})
don't have beating at the surfaces $t = \tau_j(u);$ in other words, solutions intersect each surface no more than once.
For impulsive systems in the finite-dimensional case, there are several sufficient conditions that allow us to
exclude the phenomenon of pulsation (see, \cite{SP}, \cite{SPT}). In infinite dimensional case analogous conditions cannot easily be verified.
In every concrete case one needs a separate investigation.

We assume that the corresponding linear homogeneous equation has an exponential dichotomy.
The definition of exponential dichotomy for an impulsive evolution equation corresponds to the definition of
exponential dichotomy for continuous evolution equations in an infinite-dimensional Banach space \cite{ChL,H,PS}.
We require that only solutions of a linear system from an unstable manifold can be
unambiguously extended to the negative semiaxis.

Robustness is an impotent property of the exponential dichotomy \cite{ChL,KRT,PS}. We mention the papers \cite{BV,NP,T1,T2}
where the robustness of the exponential dichotomy for impulsive systems by small perturbations of right-hand sides is proved.
In this chapter we prove robustness of the exponential dichotomy also by the small perturbation of points of impulsive
action. We use a change of time in the system.
Then approximation of the impulsive system by difference systems (see \cite{H}) can be used.
If a linear homogeneous equation is exponentially stable, we prove stability of the almost periodic solution
of nonlinear equation (\ref{ban1}), (\ref{ban2}). Following \cite{RT}, we use the generalized Gronwall inequality,
taking into account singularities in integrals and impulsive influences.

This chapter is organized as follows. In Sect.7.2 we present some preliminary
definitions and results. In Sect.7.3, we study an exponential dichotomy of impulsive
linear equations. Section 7.4 is devoted to studying the existence and stability of
almost periodic solutions in linear inhomogeneous equations with impulsive action
and semilinear impulsive equations with fixed moments of impulsive action. In
Sect.7.5 we consider impulsive evolution equations with nonfixed moments of
impulsive action. In Sect.7.6 we discuss the case of unbounded operators $B_j$ in linear
parts of impulsive action.


\section{Preliminaries}

Let $(X, \|.\|)$ be an abstract Banach space and ${R}$ and ${Z}$ be the sets of real
and integer numbers, respectively.

We will consider the space $\mathcal{PC}(J,X), \ J \subset {R},$
of all piecewise continuous functions $x: J \to X$ such that

i) the set $\{ \tau_j \in J: \tau_{j+1} > \tau_j, j \in {Z}\}$ of discontinuities of $x$ has no finite limit points;

ii) $x(t)$ is left-continuous $x(\tau_j - 0) = x(\tau_j)$ and there exists $\lim_{t \to \tau_j + 0} x(t) = x(\tau_j + 0).$

We will use the norm $\|x\|_{PC} = \sup_{t \in J}\|x(t)\|$ in the space $\mathcal{PC}(J,X)$.

\begin{definition} \label{def1}
The integer $p$ is called an $\varepsilon$-almost period of a sequence $\{x_k\}$ if
$\|x_{k+p} - x_k\| < \varepsilon$ for any $k \in {Z}.$ The sequence $\{x_k\}$ is almost periodic if for any $\varepsilon > 0$
there exists a relatively dense set of its $\varepsilon$-almost periods.
\end{definition}

\begin{definition} \label{def2}
The strictly increasing sequence $\{\tau_k\}$ of real numbers has uniformly almost periodic sequences of differences if for any $\varepsilon > 0$
there exists a relatively dense set of $\varepsilon$-almost periods common for all sequences
$\{\tau^j_k\},$ where $\tau^j_k = \tau_{k+j} - \tau_{k}, j \in {Z}.$
\end{definition}

By Samoilenko and Trofimchuk \cite{STr}, the sequence $\{\tau_k\}$ has uniformly almost periodic sequences of differences if and only if $\tau_k = a k + c_k,$
where $\{c_k\}$ is an almost periodic sequence and $a$ is a positive real number.
\vspace{2mm}

By Lemma 22 (\cite{SP}, p. 192), for a sequence $\{\tau_j\}$ with uniformly almost periodic sequences of differences there exists the
limit
\begin{eqnarray} \label{limp}
\lim_{T \to \infty} \frac{{i}(t,t+T)}{T} = p
\end{eqnarray}
uniformly with respect to $t \in {R},$ where ${i}(s, t)$ is the number of the points $\tau_k$
lying in the interval $(s, t).$ Then for each $q > 0$ there exists a positive integer $N$
such that on each interval of length  $q$ there are no more then $N$ elements of the sequence $\{\tau_j\}$;
that is, ${i}(s, t) \le N(t-s)/q + N.$

Also for sequence $\{\tau_j\}$ with uniformly almost periodic sequences of differences there exists $\Theta > 0$
such that $\tau_{j+1} - \tau_j \le \Theta, j \in {Z}.$

\begin{definition} \label{def3}
The function $\varphi \in \mathcal{PC}({R},X)$ is said to be W-almost periodic if

i) the strictly increasing sequence $\{\tau_k\}$ of discontinuities of $\varphi(t)$ has uniformly almost periodic sequences of differences;

ii) for any $\varepsilon > 0$ there exists a positive number $\delta = \delta(\varepsilon)$ such that if the points
$t'$ and $t''$ belong to the same interval of continuity and $|t' - t''| < \delta$ then $\|\varphi(t') - \varphi(t'')\| < \varepsilon;$

iii) for any $\varepsilon > 0$ there exists a relatively dense set $\Gamma$ of $\varepsilon$-almost periods such that if
$\tau \in \Gamma,$ then $\|\varphi(t +\tau) - \varphi(t)\| < \varepsilon$ for all $t \in {R}$ that satisfy
the condition $|t - t_k| \ge \varepsilon, k \in {Z}.$
\end{definition}
\vspace{1mm}

We consider the impulsive equation (\ref{ban1}), (\ref{ban2}) with the following assumptions:

${\bf (H1)}$ $A$ is a sectorial operator acting in $X$ and $\inf\{ Re \mu: \ \mu \in \sigma(A)\} \ge \delta > 0,$
where $\sigma(A)$ is the spectrum of $A.$
Consequently, the fractional powers of $A$ are well defined, and one can consider
the spaces $X^\alpha = D(A^\alpha)$ for $\alpha \ge 0$ endowed with the norms $\|x\|_\alpha = \|A^\alpha x\|$.
\vspace{1mm}

${\bf (H2)}$ The function $A_1(t): {R} \to L(X^\alpha,X)$ is Bohr almost periodic and  H\"{o}lder continuous, $\alpha \ge 0,$
$L(X^\alpha,X)$ is the space of linear bounded operators $X^\alpha \to X$.
\vspace{1mm}

${\bf (H3)}$ We shell use the notation $U^\alpha_\varrho = \{ x \in X^\alpha: \ \| x\|_\alpha \le \varrho\}.$
Assume that the sequence $\{\tau_j(u)\}$ of functions $\tau_j: U^\alpha_\varrho \to {R}$
has uniformly almost periodic sequences of differences
uniformly with respect to $u \in U^\alpha_\varrho$ and there exists $\theta > 0$ such that $\inf_{u} \tau_{j+1}(u) -
\sup_{u} \tau_{j}(u) \ge \theta$ for all $u \in U^\alpha_\varrho$ and $j \in {Z}.$
Also, there exists $\Theta > 0$
such that $\sup_u\tau_{j+1}(u) - \inf_u\tau_j(u) \le \Theta$ for all $j \in {Z}$ and $u \in U^\alpha_\varrho.$
\vspace{1mm}

${\bf (H4)}$ The sequence $\{B_j\}$ of bounded operators is almost periodic and there exists $b > 0$ such that
$\|B_j u\|_\alpha \le b \|u\|_\alpha$ for $j \in {Z}, \alpha \ge 0$ and $u \in X^\alpha.$
\vspace{1mm}

${\bf (H5)}$ The function $f(t,u): \ {R} \times U^\alpha_\rho \to X$ is continuous in $u$ and is  locally H\"{o}lder continuous 
and W-almost periodic in $t$ uniformly with respect to $u \in U^\alpha_\rho.$ 

${\bf (H6)}$  The sequence $\{g_j(u)\}$ of continuous functions $U^\alpha_\rho \to X^\alpha$ is almost periodic
uniformly with respect to $u \in U^\alpha_\rho.$
\vspace{1mm}

\begin{remark}
Assumption ${\bf (H4)}$ is satisfied if, for example, $B_j A = A B_j$ for all $j \in Z.$
We assume that operators $B_j$ are bounded. Many of our results are valid if the
$B_j$ are unbounded closed operators $X^{\alpha+\gamma} \to X^\alpha$ for $\alpha \ge 0$ and some $\gamma > 0.$
We discuss this case in the last section.
\end{remark}

We use the following generalization of Lemma 7 from \cite{HW}, p. 288 (also, see \cite{HPTT} and \cite{SPA}):

\begin{lemma} \label{lem1}
Assume that a sequence of real numbers $\{\tau_j\}$ has uniformly almost periodic sequences of differences,
the sequence $\{B_j\}, B_j \in X,$ is almost periodic and the function
$f(t): {R} \to X$ is W-almost periodic.
Then for any $\varepsilon > 0$ there exists a such $l = l(\varepsilon) > 0$
that for any interval $J$ of length $l$ there are such $r \in J$ and an integer $q$ that the following relations hold:
$$|\tau_{i+q} - \tau_i - r| < \varepsilon, \ \| B_{i+q} - B_i\| < \varepsilon, \ i \in {Z},$$
$$\|f(t+r) - f(t)\| < \varepsilon, \
t \in R, \ |t - \tau_j| > \varepsilon, \ j \in Z.$$
\end{lemma}

If $A$ is a sectorial operator then $(-A)$ is an infinitesimal generator of the analytical semigroup $e^{-At}.$
For every $x \in X^\alpha$ we get $e^{-At} A^\alpha x = A^\alpha e^{-At} x.$
Further, we shall use the inequalities (see \cite{H})
\begin{eqnarray*}
& & \| A^\alpha e^{-A t} \| \le C_\alpha t^{-\alpha} e^{-\delta t}, \ t > 0, \ \alpha > 0, \\
& & \| (e^{-A t} - I)u \| \le \frac{1}{\alpha}C_{1-\alpha} t^\alpha \|  A^\alpha u \|, \ t > 0, \ \alpha \in (0,1], \ u \in X^\alpha,
\end{eqnarray*}
where $C_\alpha \in {R}$ is nonnegative and bounded as $\alpha \to +0.$

\begin{definition} \label{def4}
The function $u(t): [t_0, t_1] \to X^\alpha$ is said to be a solution of the initial value problem $u(t_0) = u_0 \in X^\alpha$
for Eq. (\ref{ban1}), (\ref{ban2}) on $[t_0, t_1]$ if

(i) it is continuous in $[t_0, \tau_k], (\tau_k, \tau_{k+1}],..., (\tau_{k+s}, t_1]$ with the discontinuities of the first
kind at the moments $t = \tau_j$ of intersections with impulsive surfaces;

(ii) $u(t)$ is continuously differentiable in each of the intervals $(t_0, \tau_k),$ $(\tau_k, \tau_{k+1}),$ $..., (\tau_{k+s}, t_1)$
and satisfies Eqs. (\ref{ban1}) and (\ref{ban2}) if $t \in (t_0,t_1), t \not= \tau_j,$ and $t = \tau_j,$
respectively;

(iii) the initial-value condition $u(t_0) =u_0$ is fulfilled.
\end{definition}

We assume that solutions $u(t)$ of (\ref{ban1}), (\ref{ban2}) are left-hand-side continuous,
hence $u(\tau_j) = u(\tau_j - 0)$ at all points of impulsive action.

Also we assume that in the domain $U^\alpha_\rho$ solutions of (\ref{ban1}) and (\ref{ban2})
don't have beating at the surfaces $t = \tau_j(u);$ in other words, solutions intersect each surface no more then once.

\section{Exponential Dichotomy}

Together with Eq. (\ref{ban1}), (\ref{ban2}) we consider the corresponding linear homogeneous equation
\begin{eqnarray} \label{lin1}
& & \frac{du}{dt} + (A + A_1(t)) u = 0, \quad t \not= \tau_j, \\
& & \Delta u|_{t=\tau_j} = u(\tau_j + 0) - u(\tau_j) = B_j u(\tau_j), \quad j \in Z, \label{lin2}
\end{eqnarray}
where $\tau_j = \tau_j(0).$
Denote by $V(t,s)$ the evolution operator of the linear equation without impulses (\ref{lin1}).
It satisfies $V(\tau,\tau) = I, \ V(t,s)V(s,\tau) = V(t,\tau), \ t\ge s \ge \tau.$

By Theorem 7.1.3 (\cite{H}, p.190), $V(t,\tau)$
is strongly continuous with values in $L(X^\beta)$ for any $0 \le \beta < 1$ and
\begin{eqnarray} \label{evop1}
\| V(t,\tau)x\|_\beta \le L_Q(t - \tau)^{(\gamma - \beta)_-}\|x\|_\gamma,
\end{eqnarray}
where $(\gamma - \beta)_- = \min(\gamma - \beta, 0), \ t - \tau \le Q, \ L_Q = L_Q(Q).$
Moreover,
\begin{eqnarray} \label{evop2}
\| V(t,\tau)x - x\|_\beta \le L_{Q}(t - \tau)^{\nu}\|x\|_{\beta+\nu}, \quad \nu > 0, \ \beta + \nu \le 1.
\end{eqnarray}

Using the proof of Lemma 7.1.1 from \cite{H}, p. 188, one can verify the following
generalized Gronwall inequality
\begin{lemma} \label{lem-gr}
Suppose $0 \le \alpha, \beta < 1, a_1 \ge 0, a_2 \ge 0, b \ge 0,  0 < Q < \infty$ and $y(t)$ is nonnegative function locally integrable on $0 \le t < Q$
with
$$y(t) \le a_1 + a_2 t^{-\alpha} + b \int_0^t (t-s)^{-\beta}y(s)ds$$
on this interval; then there is a constant $\tilde C = \tilde C(\beta, b, Q) < \infty$ such that
\begin{eqnarray} \label{evop22}
y(t) \le \left( a_1 + \frac{a_2}{(1-\alpha)t^\alpha}\right) \tilde C(\beta, b, Q).
\end{eqnarray}
\end{lemma}
Note that inequality (\ref{evop22}) can be rewritten as
\begin{eqnarray} \label{evop222}
y(t) \le \left( a_1 + \frac{a_2}{t^\alpha}\right) \tilde C_1, \ \tilde C_1 = \frac{\tilde C(\beta, b, Q)}{1-\alpha}.
\end{eqnarray}

We will use the following perturbation lemma.

\begin{lemma} \label{lem2}
Let us consider the perturbed equation
\begin{eqnarray} \label{evol3}
 \frac{du}{dt} + (\gamma A + A_2(t)) u = 0,
\end{eqnarray}
where $\gamma = Const > 0, \ A_2(t): {R} \to L(X^\alpha, X).$

Then for $Q > 0,$ there exists $\varepsilon_0 > 0$ such that for all $\varepsilon \le \varepsilon_0$
and $|\gamma - 1| \le \varepsilon, \ \sup_{t}\|A_1(t) - A_2(t)\|_{L(x^\alpha,X)} \le \varepsilon$
the evolution operators $V(t,s)$ of (\ref{lin1}) and $V_1(t,s)$ of (\ref{evol3}) satisfy
\begin{eqnarray} \label{evol4}
\|V(t,s) - V_1(t,s)\|_\alpha \le R_1(\varepsilon), \ t - s \le Q,
\end{eqnarray}
with $R_1(\varepsilon)$ depends on $Q, \alpha,$ and $R_1(\varepsilon) \to 0$ as $\varepsilon \to 0.$
\end{lemma}

\begin{proof} For definiteness let $\gamma > 1.$
Solutions $x(t)$ and $y(t)$ of Eqs. (\ref{lin1}) and (\ref{evol3}) satisfy the following integral equations
\begin{eqnarray*}
x(t) = e^{-A(t-t_0)}x_0 + \int_{t_0}^t e^{-A(t-s)}A_1(s)x(s)ds
\end{eqnarray*}
and
\begin{eqnarray*}
y(t) = e^{-A\gamma(t-t_0)}x_0 + \int_{t_0}^t e^{-A\gamma(t-s)}A_2(s)y(s)ds.
\end{eqnarray*}
Then
\begin{eqnarray*}
& & \|x(t) - y(t)\|_\alpha \le \|(I - e^{-A(\gamma -1)(t-t_0)})A^{\alpha}e^{-A(t-t_0)}x_0\| \\
& & + \int_{t_0}^t \|(I - e^{-A(\gamma -1)(t-s)})A^{\alpha}e^{-A(t-s)}A_1(s)x(s)\|ds  \\
& & + \int_{t_0}^t \|A^{\alpha}e^{-A\gamma(t-s)}(A_1(s) - A_2(s))x(s)\|ds  \\
& & + \int_{t_0}^t \|A^{\alpha}e^{-A\gamma(t-s)}A_2(s)(x(s) - y(s))\|ds  \\
& & \le a_1(\varepsilon)\|x_0\|_\alpha + a_2 \int_{t_0}^t (t - s)^{-\alpha}\|x(s) - y(s)\|_\alpha ds,
 \end{eqnarray*}
where $a_2 = C_{\alpha}\sup_s\|A_1(s)\|_{L(X^\alpha,X)}$ and
$a_1(\varepsilon) \to 0$ as $\varepsilon \to 0.$
By Lemma \ref{lem-gr}, there exists positive constant $K_1$ depending on $\alpha$ and $Q$ such that
$$\|x(t) - y(t)\|_\alpha \le K_1 a_1(\varepsilon)\|x_0\|_\alpha = R_2(\varepsilon)\|x_0\|_\alpha.$$
\end{proof}

\begin{lemma} \label{lem2a}
Let us consider Eq. (\ref{lin1}) and
\begin{eqnarray} \label{lin11}
\frac{dv}{dt} + (A + A_2(t)) v = 0,
\end{eqnarray}
such that $A_2: {R} \to L(X^\alpha, X)$ is a bounded and H\"{o}lder continuous function.

Then for $Q > 0,$  there exists $\varepsilon_0 > 0$ such that for all $ \varepsilon \le \varepsilon_0$
and
$$\sup_{t}\|A_1(t) - A_2(t)\|_{L(X^\alpha, X)} \le \varepsilon$$
the evolution operators $V(t,s)$ of (\ref{lin1}) and $V_1(t,s)$ of  (\ref{lin11}) satisfy
\begin{eqnarray} \label{evol44}
\|(V(t,s) - V_1(t,s))u\|_\alpha \le R_3(\varepsilon)|t-t_0|^{1-2\alpha + \delta} \|u\|_\delta, \ t - s \le Q,
\end{eqnarray}
with $R_3(\varepsilon) = R_3(\varepsilon,Q,\alpha)$ and $R_3(\varepsilon) \to 0$ as $\varepsilon \to 0.$
\end{lemma}

\begin{proof}
Denote by $u(t)$ and  $v(t)$ solutions of (\ref{lin1}) and (\ref{lin11}) with initial value $u(t_0) = v(t_0) = u_0.$
They satisfy inequalities
\begin{eqnarray}
& & \|u(t) - v(t)\|_\alpha \le \int_{t_0}^t \|A^\alpha e^{-A(t-s)}(A_1(s) - A_2(s))u(s)\| ds  \nonumber\\
& & + \int_{t_0}^t \|A^\alpha e^{-A(t-s)}A_2(s)(u(s) - v(s))\| ds  \nonumber\\
& & \le C_\alpha L_Q \varepsilon\|u_0\|_\delta \int_{t_0}^t \frac{ds}{(t-s)^\alpha (s-t_0)^{\alpha-\delta}} +
C_\alpha \|A_1\|_L \int_{t_0}^t  \frac{\|u(s) - v(s)\|_\alpha ds}{(t-s)^{\alpha}} \nonumber\\
& &  \le \varepsilon\|u_0\|_\delta R_4 + C_\alpha \|A_1\|_L \int_{t_0}^t  \frac{\|u(s) - v(s)\|_\alpha ds}{(t-s)^{\alpha}}.
\label{1ev}
\end{eqnarray}

Applying Lemma \ref{lem-gr} to (\ref{1ev}), we obtain (\ref{evol44}).

\end{proof}

We define the evolution operator for equation (\ref{lin1}), (\ref{lin2}) as
\begin{eqnarray*}
U(t,s) = V(t, s) \ {\rm if} \ \tau_k < s \le t \le \tau_{k+1}
\end{eqnarray*}
and
\begin{eqnarray} \label{evol1}
U(t,s) = V(t, \tau_k)(I + B_k)V(\tau_k, \tau_{k-1}) ... (I + B_m)V(\tau_m,s),
\end{eqnarray}
if $ \tau_{m-1} < s \le \tau_m < \tau_{m+1} < ... < \tau_k < t \le \tau_{k+1}.$

It is easy to verify that for fixed $t > s$ the operator $U(t,s)$ is bounded in the space $X^\alpha.$

\begin{definition}\label{defn:dich} 
We say that the equation  (\ref{lin1})--(\ref{lin2}) has an exponential dichotomy on ${R}$
 with exponent $\beta > 0$ and bound $M \ge 1$ (with respect to the space $X^\alpha$) if there exist projections $P(t), t\in {R},$
such that

(i) $U(t,s)P(s) = P(t)U(t,s), \ t \ge s$;
\vspace{1mm}

(ii) $U(t,s)|_{Im(P(s))}$ for  $t \ge s$ is an isomorphism
on $Im(P(s))$, and then
$U(s,t)$ is defined as an inverse map from $Im(P(t))$
to $Im(P(s))$;
\vspace{1mm}

(iii) $\|U(t,s)(1 - P(s))u\|_\alpha \le M e^{-\beta(t-s)}\|u\|_\alpha, \ t \ge s, \ u \in X^\alpha$;
\vspace{1mm}

(iv) $\|U(t,s)P(s)\|_\alpha \le M e^{\beta(t-s)}\|u\|_\alpha, \ t \le s, \ u \in X^\alpha$.
\end{definition}

\vspace{1mm}

 If Eq. (\ref{lin1})--(\ref{lin2}) has an exponential dichotomy on ${R},$
then the nonhomogeneous equation
\begin{eqnarray} \label{inlin1}
& & \frac{du}{dt} + (A + A_1(t)) u = f(t), \quad t \not= \tau_j, \\
& & \Delta u|_{t=\tau_j} = u(\tau_j + 0) - u(\tau_j) = B_j u(\tau_j) + g_j, \quad j \in  Z, \label{inlin2}
\end{eqnarray}
has a unique solution bounded on ${R}$
\begin{equation} \label{inlin3}
u_0(t) = \int_{-\infty}^{\infty}G(t,s)f(s)(x)ds + \sum_{j \in {Z}}G(t, \tau_j + 0)g_j,
\end{equation}
where
\begin{eqnarray*}
G(t,s) =
 \left\{\begin{array}{l}
            U(t,s)(I - P(s)), \ \ t \ge s,\\
            -U(t,s)P(s), \ \ t < s,
          \end{array} \right.
\end{eqnarray*}
is the Green function such that
\begin{eqnarray} \label{inlin5}
\|G(t,s)u\|_\alpha \le M e^{-\beta|t-s|}\|u\|_\alpha, \ t,s \in {R}.
\end{eqnarray}

Analogous to \cite{H}, p.250, it can be proven that a function $u(t)$ is a bounded solution of (\ref{inlin1}), (\ref{inlin2})
on the semiaxis $[t_0, +\infty)$ if and only if \ $u(t) =$
\begin{eqnarray*} 
= U(t,t_0)(I-P(t_0))u(t_0) + \int_{t_0}^{+\infty} G(t,s)f(s)ds + \sum_{t_0 \le \tau_j} G(t,\tau_j+0)g_j, \ t \ge t_0.
\end{eqnarray*}

A function $u(t)$ is bounded solution
on the semiaxis $(-\infty, t_0]$ if and only if
\begin{eqnarray*} 
 u(t) = U(t,t_0)P(t_0)u(t_0) + \int^{t_0}_{-\infty} G(t,s)f(s)ds + \sum_{t_0 > \tau_j} G(t,\tau_j+0)g_j, \ t \le t_0.
\end{eqnarray*}

Now we estimate $\| G(t,s)u\|_\alpha$ for $u \in X.$ Let $t > s$ and $\tau_{m-1} < s \le \tau_{m}, \ \tau_{k} < t \le \tau_{k+1}.$
Then
\begin{eqnarray} \label{gr1}
& & \|G(t,s)u\|_\alpha = \|U(t,s)(I - P(s))u\|_\alpha  \nonumber\\
& & \le \|U(t,\tau_m)(I - P(\tau_m))\|_\alpha \| U(\tau_m,s) u\|_\alpha  \nonumber\\
& & \le M e^{-\beta(t - \tau_m)} L_\Theta (\tau_m - s)^{-\alpha}\|u\| \le \tilde M e^{-\beta(t-s)}|\tau_m - s|^{-\alpha}\|u\|
\end{eqnarray}
and
\begin{eqnarray}
& & \| G(s,t)u\|_\alpha = \| U(s,t)P(t)u\|_\alpha  \nonumber\\
& & \le \|U(s,t+1)P(t+1)\|_\alpha \| A^\alpha U(t+1,t) u\| \le  \tilde M e^{-\beta(t-s)}\|u\|. \label{gr2}
\end{eqnarray}

If $t_1$ and $t_2$ belong to the same interval of continuity, then
\begin{eqnarray} \label{gr3}
\|P(t_1)u - P(t_2)u\|_\gamma \le \tilde M_1 \|t_1 - t_2|^\nu \|u\|_{\gamma + \nu}
\end{eqnarray}
since as in \cite{H}, p.247,
\begin{eqnarray*}
& & \|P(t + h)u - P(t)u\|_\gamma \le  \|P(t)u - V(t+h,t)P(t)u\|_\gamma  \\
& & + \|V(t+h,t)P(t)u - P(t+h)u\|_\gamma  \\
& & \le \|(I - V(t+h,t))P(t)u\|_\gamma + \|P(t+h)(V(t+h,t)u - u)\|_\gamma.
\end{eqnarray*}

\begin{lemma} \label{lem4}
Let the impulsive equation (\ref{lin1}), (\ref{lin2}) is exponentially dichotomous
with positive constants $\beta$ and $M.$
Then there exists $\varepsilon > 0$ such that the perturbed equations
\begin{eqnarray} \label{lin21}
& & \frac{du}{dt} + (A + \tilde A(t))u = 0, \quad t \not= \tilde\tau_j, \\
& & \Delta u|_{t=\tilde \tau_j} = u(\tilde\tau_j + 0) - u(\tilde\tau_j) = \tilde B_j u(\tilde\tau_j), \quad j \in  Z, \label{lin22}
\end{eqnarray}
with $\sup_j |\tau_j - \tilde \tau_j| \le \varepsilon, \sup_j \|B_j -\tilde B_j\| \le \varepsilon, \
\sup_t \|A_1(t) - \tilde A(t)\|_{L((X^\alpha,X)} \le \varepsilon,$
are also exponentially dichotomous with some constants $\beta_1 \le \beta$ and $M_1 \ge M.$
\end{lemma}

{\bf Proof.} In Eq. (\ref{lin1}), (\ref{lin2}), we introduce the change of time
$t = \vartheta(t')$ such that $\tau_j = \vartheta(\tilde\tau_j), j \in {Z},$
and the function $\vartheta$ is continuously differentiable and monotonic 
on each interval $(\tilde\tau_j, \tilde\tau_{j+1}).$

The function $\vartheta$ can be chosen in piecewise linear form
\begin{eqnarray} \label{lin23}
t = a_j t' + b_j, \ a_j = \frac{\tau_{j+1} - \tau_{j}}{\tilde\tau_{j+1} - \tilde\tau_{j}}, \
b_j = \frac{\tau_{j}\tilde\tau_{j+1} - \tau_{j+1}\tilde\tau_{j}}{\tilde\tau_{j+1} - \tilde\tau_{j}} \ \ {\rm if} \ \ t' \in (\tilde\tau_{j}, \tilde\tau_{j+1}).
\end{eqnarray}

The function $\vartheta(t')$ satisfies the conditions
$$|\vartheta(t') - t'| \le \varepsilon,  \ |\frac{d\vartheta(t')}{dt'} - 1| \le 2\varepsilon/\theta.$$

The equation (\ref{lin1}), (\ref{lin2}) in the new coordinates $v(t') = u(\vartheta(t'))$ has the form
\begin{eqnarray} \label{lin221}
& & \frac{dv}{dt'} + \frac{d\vartheta(t')}{d t'}\left(A + A_1(\vartheta(t')\right)v = 0, \quad t \not= \tilde\tau_j, \\
& & \Delta v|_{t'= \tilde\tau_j} = v(\tilde\tau_j + 0) - v(\tilde\tau_j) =  B_j v(\tilde\tau_j), \quad j \in  Z. \label{lin222}
\end{eqnarray}
Eq. (\ref{lin221}), (\ref{lin222}) has the evolution operator $U_1(t',s') = U(\vartheta(t'), \vartheta(s')).$
If Eq. (\ref{lin1}), (\ref{lin2}) has an exponential dichotomy
with projector $P(t)$ at point $t$, then Eq. (\ref{lin221}), (\ref{lin222}) has
an exponential dichotomy with projector $P_1(t') = P(\vartheta(t'))$ at point $t'$. Really,
\begin{eqnarray*}
& & \|U_1(t',s')(1 - P_1(s'))\|_\alpha =  \|U(\vartheta(t'),\vartheta(s'))(1 - P(\vartheta(s'))\|_\alpha  \\
& & \le M e^{-\beta(\vartheta(t')- \vartheta(s'))} \le  M e^{2\varepsilon} e^{-\beta(t'-s')}, \ t \ge s.
\end{eqnarray*}
The inequality for an unstable manifold is proved analogously.

The linear equations (\ref{lin221}), (\ref{lin222}) and (\ref{lin21}), (\ref{lin22}) have the same points of impulsive actions
$\tilde\tau_j, j \in {Z},$ and
\begin{eqnarray*}
& & \|\frac{d\vartheta(t')}{d t'}A_1(\vartheta(t')) - \tilde A(t')\| \le \|\frac{d\vartheta(t')}{d t'}A_1(\vartheta(t')) - A_1(\vartheta(t'))\|  \\
& & + \|A_1(\vartheta(t')) - A_1(t')\| +  \|A_1(t') - \tilde A(t')\| \le K_2(\varepsilon),
\end{eqnarray*}
where $K_2(\varepsilon) \to 0$ as $\varepsilon \to 0.$

Let $\tilde U(t',s')$ be the evolution operator for Eq. (\ref{lin21}), (\ref{lin22}).
To show that for sufficiently small $\delta_0 > 0$ Eq. (\ref{lin21}), (\ref{lin22}) is exponentially dichotomous,
we use the following variant of Theorem 7.6.10, \cite{H}:

Assume that the evolution operator $U_1(t',s')$ has an exponential dichotomy on
 ${R}$ and satisfies
 \begin{eqnarray}\label{evol_bound}
\sup\limits_{0\le t'-s'\le d}\|U_1(t',s')\|_\alpha < \infty
\end{eqnarray}
for some positive $d.$
Then there exists $\eta > 0$ such that
$$
\|\tilde U(t',s') - U_1(t',s')\|_\alpha < \eta, \ \ {\rm whenever} \ \ t-s \le d;
$$
the evolution operator $\tilde U(t',s')$ also has an exponential dichotomy on
 ${R}$ with some constants $\beta_1 \le \beta, M_1 \ge M.$

To prove this statement, we set for $n \in {Z}$
 $$ t_n = s' + dn, \ T_n = U_1(s' + d(n+1), s' + dn+0), \
\tilde T_n = \tilde U(s' + d(n+1), s' + dn+0). $$
If the evolution operator $U_1(t,s)$ has an exponential dichotomy,
then  $\left\{T_n\right\}$ has a discrete dichotomy in the sense of \cite[Definition 7.6.4]{H}.

According to Henry \cite{H}, Theorem 7.6.7, there exists $\eta> 0$ such that $\{\tilde T_n\}$ with
$\sup_n\|T_n - \tilde T_n\|_\alpha  \le \eta$
 has a discrete dichotomy.

Now we are in the conditions of \cite{H},Exercise 10, p. 229--230 (see also a more general statement \cite{ChL}, Theorem 4.1),
what finishes the proof.

Let us estimate the difference $\|\tilde T_k - T_k\|_\alpha.$ There exists a positive integer $N$ such that each interval of length $d$
contains no more then $N$ elements of sequence $\{\tau_j\}.$
Let the interval $(\xi_{n}, \xi_{n+1}]$ contains points of impulses
$\tilde\tau_m,..., \tilde\tau_k$ where $k-m \le N.$
Denote by $V_1(t,s)$ and $\tilde V(t,s)$ the evolution operators of equations without impulses (\ref{lin221}) and (\ref{lin21}),
respectively. Then
\begin{eqnarray}
& & \| T_n- \tilde T_n\|_\alpha  = \|U_1(\xi_{n+1},\xi_n)- \tilde U(\xi_{n+1},\xi_n)\|_\alpha   \nonumber \\
& & \le \|(V_1(\xi_{n+1}, \tilde\tau_k) - \tilde V(\xi_{n+1}, \tilde\tau_k))(I + B_k)V_1(\tilde\tau_k, \tilde\tau_{k-1}) ...
(I + B_m)V_1(\tilde\tau_m,\xi_n)\|_\alpha  \nonumber\\
& & +  \|\tilde V(\xi_{n+1}, \tilde\tau_k)(B_k - \tilde B_k)V_1(\tilde\tau_k, \tilde\tau_{k-1}) ...
(I + B_m)V_1(\tilde\tau_m, \xi_n)\|_\alpha + \ ... \ \nonumber\\
& & + \|\tilde V(\xi_{n+1}, \tilde\tau_k)(I + \tilde B_k)\tilde V(\tilde\tau_k, \tilde\tau_{k-1}) ...
(I + \tilde B_m)(V_1(\tilde\tau_m,\xi_n) - \tilde V(\tilde\tau_m,\xi_n))\|_\alpha. \label{dd0}
\end{eqnarray}
Using (\ref{evol4}), we get that
\begin{eqnarray*}
\sup_n \| T_n
- \tilde T_n\|_\alpha \le K_3(\varepsilon)
\end{eqnarray*}
with $K_3(\varepsilon) \to 0$ as $\varepsilon \to 0.$

The exponentially dichotomous equation (\ref{lin21}), (\ref{lin22}) has Green's function
\begin{eqnarray*}
\tilde G(t,s) =
 \left\{\begin{array}{l}
            \tilde U(t,s)(I - \tilde P(s)), \ t \ge s,\\
            -\tilde U(t,s)\tilde P(s), \ t < s,
          \end{array} \right.
\end{eqnarray*}
 such that
\begin{eqnarray*}
\|\tilde G(t,s)u\|_\alpha \le M_1 e^{-\beta_1|t-s|}\|u\|_\alpha, \ t,s \in {R}, \ u \in X^\alpha.
\end{eqnarray*}

The sequence of bounded operators $T_n: X^\alpha \to X^\alpha$ defines the difference equation
\begin{eqnarray} \label{dd1}
u_{n+1} = T_n u_n, \ n \in {Z},
\end{eqnarray}
with evolution operator $T_{n,m} = T_{n-1}...T_m, \ n \ge m, \ T_{m,m} = I.$
It is exponentially dichotomous with Green's function
\begin{eqnarray*}
G_{n,m} =  \left\{ \begin{array}{l}
            T_{n,m}(I - P_m), \ n \ge m, \\
            - T_{n,m}P_m, \ n < m,
          \end{array} \right.
\end{eqnarray*}
where $P_m = P(\xi_m).$
The second difference equation
\begin{eqnarray} \label{dd2}
u_{n+1} = \tilde T_n u_n, \ n \in {Z},
\end{eqnarray}
has evolution operator $\tilde T_{n,m} = \tilde T_{n-1}...\tilde T_m, \ n \ge m, \ \tilde T_{m,m} = I.$

By sufficiently small $\sup_n \|T_n - \tilde T_n\|_\alpha,$ Eq. (\ref{dd2}) is exponentially dichotomous
with Green's function
\begin{eqnarray*}
\tilde G_{n,m} =  \left\{ \begin{array}{l}
            \tilde T_{n,m}(I - \tilde P_m), \ n \ge m, \\
            - \tilde T_{n,m}\tilde P_m, \ n < m,
          \end{array} \right.
\end{eqnarray*}

According to Henry \cite{H}, p. 233, the difference between two Green's functions satisfies equality
\begin{eqnarray} \label{dd3}
\tilde G_{n,m} - G_{n,m} =  \sum_{k \in {Z}}G_{n,k+1}(\tilde T_k - T_k)\tilde G_{k,m}
\end{eqnarray}
and estimation
\begin{eqnarray} \label{dd4}
\|\tilde G_{n,m} - G_{n,m}\|_{\alpha} =
 M_2 e^{-\beta_2 d|n-m|}\sup_k\|\tilde T_k - T_k \|_\alpha, \ n,m \in {Z},
\end{eqnarray}
with some constants $\beta_2 \le \beta_1, M_2 \ge M_1.$

Now we can consider the difference of two Green's functions $\tilde G(t,s) - G_1(t,s).$
Let $t = s+nd+ t_1, t_1 \in [0,d).$ Then
\begin{eqnarray*}
& & \|\tilde G(t,s) - G_1(t,s)\|_\alpha \\
& & = \| \tilde U(s+nd+ t_1, s+nd) \tilde G(s+nd, s) - U(s+nd+ t_1, s+nd) G(s+nd, s)\|_\alpha  \\
& & \le \|(  \tilde U(s+nd+ t_1, s+nd) - U(s+nd+ t_1, s+nd)) \tilde G(s+nd, s)\|_\alpha  \\
& & + \| U(s+nd+ t_1, s+nd)(\tilde G(s+nd, s) - G(s+nd, s))\|_\alpha.
\end{eqnarray*}
Using (\ref{dd4}) and an estimation of the difference $\tilde U - U_1$ at a bounded interval as is done in (\ref{dd0}),
we get
\begin{eqnarray} \label{sm8}
\|\tilde G(t,\tau) - G_1(t,\tau)\|_\alpha \le \tilde M_2(\varepsilon) e^{-\beta_2|t-\tau|}, \quad t, \tau \in {R},
\end{eqnarray}
with $\tilde M_2(\varepsilon) \to 0$ as $\varepsilon \to 0.$

By the definition of Green's function, we have
\begin{eqnarray} \label{sm9}
\|\tilde P(\tau) - P_1(\tau)\|_\alpha \le \tilde M_2(\varepsilon) \ \ {\rm for \ all} \ \ \tau \in {R}.
\end{eqnarray}

\begin{corollary}
Let the conditions of Lemma \ref{lem4} be satisfied. Then for $t \in {R}, |t - \tau_j| \ge \varepsilon, j \in {Z},$ we have
\begin{eqnarray} \label{sm10}
\|(P(t) - \tilde P(t))u\|_\alpha \le \tilde M_3(\varepsilon)\|u\|_{\alpha + \nu},
\end{eqnarray}
where $\nu > 0, \alpha + \nu < 1,$ and $\tilde M_3(\varepsilon) \to 0$ as $\varepsilon \to 0.$
\end{corollary}
\begin{proof}
  Using (\ref{gr3}) and (\ref{sm9}), we get
\begin{eqnarray*}
& & \|(P(t) - \tilde P(t))u\|_\alpha \le \|(P(t) - P(\vartheta(t)))u\|_\alpha  \\
& & + \|(P(\vartheta(t)) - \tilde P(\vartheta(t)))u\|_\alpha + \|(\tilde P(\vartheta(t)) - \tilde P(t))u\|_\alpha \le
\tilde M_3(\varepsilon)\|u\|_{\alpha + \nu}.
\end{eqnarray*}
 \end{proof}


\section{Almost Periodic Solutions of Equations with Fixed Moments of Impulsive Action }

Consider the linear inhomogeneous equation
\begin{eqnarray} \label{in1}
& & \frac{du}{dt} + (A + A_1(t))u = f(t), \quad t \not= \tau_j, \\
& & \Delta u|_{t=\tau_j} = u(\tau_j+0) - u(\tau_j) = B_j u(\tau_j) + g_j, \quad j \in  Z. \label{in2}
\end{eqnarray}

We assume that

${\bf (H7)}$ the function $f(t): {R} \to X$ is W-almost periodic and locally H\"{o}lder continuous with points of discontinuity
at moments $t = \tau_j, j \in {Z},$ at which it is continuous from the left;
\vspace{1mm}

${\bf (H8)}$ the sequence $\{g_j\}$ of $g_j \in X^{\alpha_1}, \alpha_1  >  \alpha > 0,$ is almost periodic.
\vspace{2mm}

\begin{theorem} \label{thm1}
Assume that  Eq. (\ref{in1}), (\ref{in2}) satisfy conditions ${\bf (H1)}$ -- ${\bf (H3)}$, ${\bf (H7)},$ and ${\bf (H8)}$  and
that the corresponding homogeneous equation is exponentially dichotomous.

Then the equation has a unique W-almost periodic solution $u_0(t) \in \mathcal{PC}({R}, X^\alpha).$

\end{theorem}

\begin{proof}  We show that an almost periodic solution is given by the formula (\ref{inlin3}).
For $t \in (\tau_i, \tau_{i+1}],$ it satisfies
\begin{eqnarray}
& & \|u_0(t)\|_\alpha \le \int_{-\infty}^{t}\|A^\alpha U(t,s)(I - P(s))f(s)\| ds  \nonumber\\
& & + \int^{\infty}_{t}\|A^\alpha U(t,s)P(s)f(s)\| ds + \sum_{j \in {Z}} \| G(t, \tau_j + 0)g_j\|_\alpha  \nonumber\\
& & \le   \sum_{j \in{Z}} \| G(t, \tau_j + 0)g_j\|_\alpha  + \int_{\tau_i}^t \|A^\alpha V(t,s)(I - P(s))f(s)\| ds  \nonumber\\
& & + \sum_{k=0}^\infty \int_{\tau_{i-k-1}}^{\tau_{i-k}} \| U(t,\tau_{i-k})(I - P(\tau_{i-k}))\|_\alpha \| A^\alpha U(\tau_{i-k},s)f(s)\|ds  \nonumber\\
& &  + \sum_{k=1}^\infty \int_{\tau_{i+k}}^{\tau_{i+k+1}} \| U(t,\tau_{i+k+1})P(\tau_{i+k+1})\|_\alpha \| A^\alpha U(\tau_{i+k+1},s)f(s)\|ds
\nonumber\\
& & + \int^{\tau_{i+1}}_t \|A^\alpha V(t,s)P(s)f(s)\| ds \le  \frac{2M}{1 - e^{-\theta\beta}} \frac{C_\alpha \Theta^{1-\alpha}}{1-\alpha}
\|f\|_{PC}  \nonumber\\
& & + \frac{2M}{1 - e^{-\theta\beta}} \sup_j \|g_j\|_\alpha \le
 \tilde M_0 \max \{\|f(t)\|_{PC}, \|{g_j}\|_\alpha\} \label{in30}
\end{eqnarray}
with some constant $\tilde M_0 > 0.$

Take an $\varepsilon$-almost period $h$ for the right-hand side of the equation, which satisfies conditions of
 Lemma 1; that is, there exists a positive integer $q$ such that $\tau_{j+q} \in (s+h, t+h)$ if $\tau_j \in (s,t)$ and
$|\tau_{j} + h - \tau_{j+q}| < \varepsilon, \|B_{j+q} - B_j\| < \varepsilon.$

Let $t \in (\tau_i + \varepsilon, \tau_{i+1} - \varepsilon).$
 We define points $\eta_k = (\tau_k + \tau_{k-1})/2, k \in {Z}$.
 Then
\begin{eqnarray}
& & \|u_0(t + h) - u_0(t)\|_\alpha \le \sum_{j \in {Z}} \| G(t + h, \tau_{j+q}+0)g_{j+q} - G(t,\tau_j+0)g_j\|_\alpha  \nonumber\\
& & + \int_{-\infty}^{\infty}\| G(t + h,s + h)f(s + h) - G(t,s)f(s)\|_\alpha ds   \nonumber\\
& & \le \int_{-\infty}^{\infty}\|(G(t + h,s + h) - G(t,s))f(s+h)\|_\alpha ds  \nonumber\\
& & + \int_{-\infty}^{\infty}\|G(t,s))(f(s+h) - f(s))\|_\alpha ds + \sum_{j \in {Z}} \|G(t,\tau_j+0))(g_{j+q} - g_j)\|_\alpha  \nonumber\\
& & + \sum_{j \in {Z}} \| (G(t + h, \tau_{j+q}+0) - G(t,\tau_j+0))g_{j+q}\|_\alpha.
\label{in40}
\end{eqnarray}

Denote $U_2(t,s) = U(t+h,s+h).$
If $u(t) = U(t,s)u_0, u(s) = u_0,$ is a solution of the impulsive equation (\ref{lin1}), (\ref{lin2}),
then $u_2(t) = U(t+h, s+h)u_0, u_2(s) = u_0,$ is a solution of the equation
\begin{eqnarray} \label{lin1s}
& & \frac{du}{dt} + (A + A_1(t+h))u = 0, \quad t \not= \tau_{j+q}-h, \\
& & \Delta u|_{t+h=\tau_{j+q}} = u(\tau_{j+q}+0) - u(\tau_{j+q}) = B_{j+q} u(\tau_{j+q}), \quad j \in  Z. \label{lin2s}
\end{eqnarray}
We will use the notation $V_2(t,s) = V(t+h, s+h)$ for the evolution operator of the equation without impulses (\ref{lin1s}).
Denote also $\tilde \tau_n = \tau_{n+q} - h, \tilde B_n = B_{n+q}.$
Since Eq. (\ref{lin1}), (\ref{lin2}) is exponentially dichotomous, Eq. (\ref{lin1s}), (\ref{lin2s})
is exponentially dichotomous also with projector $P_2(s) = P(s + h).$

The first integral in (\ref{in40}) is the sum of two integrals:
\begin{eqnarray} \label{in50}
& & \int_{-\infty}^{\infty} \|(G(t + r,s + r) - G(t,s))f(s+r)\|_\alpha ds  \nonumber\\
& & = \int_{-\infty}^{t} \|(U_2(t,s)(I - P_2(s)) - U(t,s)(I - P(s)))f(s+r)\|_\alpha ds  \nonumber\\
& & + \int^{\infty}_{t} \|(U_2(t,s)P_2(s) - U(t,s)P(s))f(s+r)\|_\alpha ds.
\end{eqnarray}

We estimate the first integral in (\ref{in50}); the second integral is considered analogously.
\begin{eqnarray} \label{in600}
& & \int_{-\infty}^{t} \|(U_2(t,s)(I - P_2(s)) - U(t,s)(I - P(s)))f(s+r)\|_\alpha ds  \nonumber\\
& & \le \int_{\tau_i+\varepsilon}^t \|A^\alpha (V_2(t,s)(I - P_2(s)) - V(t,s)(I - P(s)))f(s+r)\| ds  \nonumber\\
& & + \int_{\tau_i - \varepsilon}^{\tau_i + \varepsilon} \|A^\alpha (U_2(t,s)(I - P_2(s)) - U(t,s)(I - P(s)))f(s+r)\| ds  \nonumber\\
& & + \int_{\eta_i}^{\tau_i - \varepsilon} \|A^\alpha(U_2(t,s)(I - P_2(s)) - U(t,s)(I - P(s)))f(s+r)\| ds  \nonumber\\
& & + \sum_{k=1}^{\infty} \int_{\eta_{i-k}}^{\eta_{i-k+1}} \|A^\alpha(U_2(t,s)(I - P_2(s)) - U(t,s)(I - P(s)))f(s+r)\| ds. \label{in6}
\end{eqnarray}

Let us consider all integrals in (\ref{in600}) separately.
By (\ref{sm10}) and (\ref{evol44}) we have
\begin{eqnarray*}
& & I_{11} = \int_{\tau_i+\varepsilon}^t \|A^\alpha (V_2(t,s)(I - P_2(s)) - V(t,s)(I - P(s)))f(s+r)\| ds  \nonumber\\
& & = \int_{\tau_i+\varepsilon}^t \|A^\alpha ((I - P_2(t))V_2(t,s) - (I - P(t))V(t,s))f(s+r)\| ds  \nonumber\\
& & \le \int_{\tau_i+\varepsilon}^t \|A^\alpha (P_2(t) - P(t))V_2(t,s)f(s+r)\| ds  \nonumber\\
& & + \int_{\tau_i+\varepsilon}^t \|A^\alpha (I - P(t))(V_2(t,s) - V(t,s))f(s+r)\| ds  \nonumber\\
& & \le \left( \int_{\tau_i+\varepsilon}^t \frac{\tilde M_3(\varepsilon) L_Q ds}{(t-s)^{\alpha}} +
 \int_{\tau_i+\varepsilon}^t \frac{ R_3(\varepsilon) ds}{(t-s)^{2\alpha-1}}\right)\|f\|_{PC} \le  \Gamma_1(\varepsilon) \|f\|_{PC}.
\end{eqnarray*}

\begin{eqnarray*}
& & I_{12} = \int_{\tau_i - \varepsilon}^{\tau_i + \varepsilon} \|A^\alpha U(t,s)(I - P(s))f(s+h)\| ds  \\
& & \le \int_{\tau_i}^{\tau_i + \varepsilon}  \|A^\alpha (I - P(t))V(t,s)f(s+h)\| ds  \\
 & & + \int_{\tau_i - \varepsilon}^{\tau_i} \|\|A^\alpha (I - P(t))V(t,\tau_i)(I + B_i)U(\tau_i,s)f(s+h)\| ds  \\
 & & \le \Bigl( \int_{\tau_i}^{\tau_i+\varepsilon} \frac{C_\alpha  ds}{(t - s)^\alpha}
M\|I+B_i\| \int_{\tau_i - \varepsilon}^{\tau_i} \frac{C_{\alpha}ds}{(s - \tau_i)^{\alpha}}\Bigl)\|f\|_{PC}  \\
& & \le \Gamma_2(\varepsilon) \|f\|_{PC}.
\end{eqnarray*}
Analogously,
\begin{eqnarray*}
& & I_{13} = \int_{\tau_i - \varepsilon}^{\tau_i + \varepsilon} \|A^\alpha U_2(t,s)(I - P_2(s))f(s+h)\| ds \le \Gamma_3(\varepsilon) \|f\|_{PC},
\end{eqnarray*}
where $\Gamma_j(\varepsilon) \to 0$ as $\varepsilon \to 0, \ j = 1,2,3.$

Using (\ref{evol44}) and (\ref{sm10}), we get
\begin{eqnarray*}
& & I_{14} = \int_{\eta_i}^{\tau_i - \varepsilon} \|A^\alpha(U_2(t,s)(I - P_2(s)) - U(t,s)(I - P(s)))f(s+r)\| ds  \\
& & = \int_{\eta_i}^{\tau_i - \varepsilon} \| \Bigl((I - P_2(t))V_2(t, \tilde\tau_i)(I + \tilde B_{i})V_1(\tilde\tau_{i},s)  \\
& & - (I - P(t))V(t,\tau_{i})(I + B_{i})V(\tau_{i},s)\Bigl)f(s+h)\|_\alpha ds  \\[2mm]
& & \le \int_{\eta_i}^{\tau_i - \varepsilon} \|(P_2(t) - P(t))V_2(t, \tilde\tau_i)(I + B_{i})V_2(\tilde\tau_i,s)f(s+h)\|_{\alpha} ds  \\
& & + \int_{\eta_i}^{\tau_i - \varepsilon} \|(I - P(t))(V_2(t, \tilde\tau_i) - V(t,\tau_{i}))(I + B_{i})V_2(\tilde\tau_i,s)f(s+h)\|_{\alpha} ds  \\
& & + \int_{\eta_i}^{\tau_i - \varepsilon} \|(I - P(t))V(t,\tau_{i})(\tilde B_{i} -B_i)V_2(\tilde\tau_i,s)f(s+h)\|_{\alpha} ds  \\
& & + \int_{\eta_i}^{\tau_i - \varepsilon} \|(I - P(t))V(t,\tau_{i})(I - B_i)(V_2(\tilde\tau_i,s) - V(\tau_{i},s))f(s+h)\|_{\alpha} ds  \\
& & \le  \Gamma_4(\varepsilon)\|f\|_{PC},
\end{eqnarray*}
where $\Gamma_4(\varepsilon) \to 0$ as $\varepsilon \to 0.$

The last sum in (\ref{in600}) is transformed as follows:
\begin{eqnarray*}
& & I_{15} = \sum_{k=1}^{\infty} \int_{\eta_{i-k}}^{\eta_{i-k+1}} \|A^\alpha(U_2(t ,s)(I - P_2(s)) - U(t,s)(I - P(s)))f(s+r)\| ds  \\
& & = \sum_{k=1}^{\infty} \int_{\eta_{i-k}}^{\eta_{i-k+1}} \|(U(t, \eta_i)(I - P(\eta_i))U(\eta_i, \eta_{i-k+1})U(\eta_{i-k+1},s)  \\
& & - U_2(t, \eta_i)(I - P_2(\eta_i))U_2(\eta_i, \eta_{i-k+1})U_2(\xi_{i-k+1},s))f(s+h)\|_{\alpha} ds  \\
& & \le \sum_{k=1}^{\infty} \int_{\eta_{i-k}}^{\eta_{i-k+1}}\Bigl\|\Bigl((U(t, \eta_i) -
U_2(t, \eta_i)) (I - P(\eta_i))U(\eta_i, \eta_{i-k+1})U(\eta_{i-k+1},s) \\
& & + U_2(t, \eta_i)\left( (I - P(\eta_i))U(\eta_i, \eta_{i-k+1}) -(I - P_2(\eta_i)) U_2(\eta_i, \eta_{i-k+1})\right)
U(\eta_{i-k+1},s)  \\
& & + U_2(t,  \eta_{i-k+1})(I - P_2( \eta_{i-k+1}))
\left(U(\eta_{i-k+1},s) - U_2(\eta_{i-k+1},s)\right)\Bigl)f(s+h)\Bigl\|_\alpha ds.
\end{eqnarray*}

As in the proof of Lemma \ref{lem4}, we construct in space $X^\alpha$ two sequences of bounded operators
$$S_n = U(\eta_{n+1}, \eta_n), \quad \tilde S_n = U_2(\eta_{n+1}, \eta_n), \quad n \in {Z},$$
and corresponding difference equations
\begin{eqnarray*}
u_{n+1} = S_n u_n, \quad v_{n+1} = \tilde S_n v_n, \quad n \in {Z}.
\end{eqnarray*}

Per our assumption, these difference equations are exponentially dichotomous
with corresponding evolution operators
$$S_{n,m} = S_{n-1}...S_m, \quad \tilde S_{n,m} = \tilde S_{n-1}...\tilde S_m, \quad n \ge m,$$
and Green's functions
\begin{eqnarray*}
G_{n,m} = \left\{ \begin{array}{l}
            S_{n,m}(I - P_m), \ n \ge m, \\
            - S_{n,m} P_m, \ n < m,
          \end{array} \right. \quad
\tilde G_{n,m} =  \left\{ \begin{array}{l}
            \tilde S_{n,m}(I - \tilde P_m), \ n \ge m, \\
            - \tilde S_{n,m}\tilde P_m, \ n < m,
          \end{array} \right.
\end{eqnarray*}
where $P_m = P(\eta_m), \tilde P_m = P_2(\eta_m).$

Analogous to (\ref{dd3}) and (\ref{dd4}), we obtain
\begin{eqnarray*} 
\tilde G_{n,m} - G_{n,m} =  \sum_{k \in {Z}}G_{n,k+1}(\tilde S_k - S_k)\tilde G_{k,m}
\end{eqnarray*}
and
\begin{eqnarray} \label{dd44}
\|\tilde G_{n,m} - G_{n,m}\|_{\alpha} =
 M_1 e^{-\beta_1 \theta |n-m|}\sup_k\|\tilde S_k - S_k \|_\alpha, \ n,m \in {Z}
\end{eqnarray}
with some constants $\beta_1 \le \beta, M_1 \ge M.$
\begin{eqnarray*}
& & \| S_n - \tilde S_n\|_\alpha  = \|U(\eta_{n+1},\eta_n)-  U_2(\eta_{n+1},\eta_n)\|_\alpha    \\
& & = \|V(\eta_{n+1}, \tau_n)(I + B_n) V(\tau_n, \eta_{n}) -
V_2(\eta_{n+1}, \tilde\tau_n)(I + \tilde B_n) V_2(\tilde\tau_n, \eta_{n})\|_\alpha   \\
& & \le \|(V(\eta_{n+1}, \tau_n) - V_2(\eta_{n+1}, \tilde \tau_n) )(I + B_n) V(\tau_n, \eta_{n}) \|_\alpha   \\
 & & + \|V_2(\eta_{n+1}, \tilde \tau_n) )(B_n - \tilde B_n) V(\tau_n, \eta_{n}) \|_\alpha   \\
 & & + \| V_2(\eta_{n+1}, \tilde \tau_n) )(I + \tilde B_n) (V(\tau_n, \eta_{n}) - V_2(\tilde\tau_n, \eta_{n})) \|_\alpha.
\end{eqnarray*}
Here we assume for definiteness that $\tilde\tau_{n} \ge \tau_{n}.$ We have
\begin{eqnarray*}
& & \|(V(\eta_{n+1}, \tau_n) - V_2(\eta_{n+1}, \tilde \tau_n))y\|_\alpha \le \| V(\eta_{n+1}, \tilde \tau_n)(V(\tilde \tau_n,\tau_n) - I)y\|_\alpha  \\
& & + \|( V(\eta_{n+1}, \tilde \tau_n) - V_2(\eta_{n+1}, \tilde \tau_n))y\|_\alpha  \\
& & \le \Gamma_5(\varepsilon)\|y\|_\alpha
\end{eqnarray*}
and
\begin{eqnarray*}
& & \|(V_2(\tilde \tau_n,\eta_{n}) - V(\tau_n,\eta_{n}))y \|_\alpha \le \| (V_2(\tilde \tau_n,\tau_{n}) - I)V_2(\tau_n,\eta_{n})y\|_\alpha \\
& & + \|V_2(\tau_n,\eta_{n}) - V(\tau_n,\eta_{n})y\|_\alpha \le  \Gamma_6(\varepsilon)\|y\|_\alpha
\end{eqnarray*}
where $\Gamma_5(\varepsilon) \to 0$ and $\Gamma_6(\varepsilon) \to 0$ as $\varepsilon \to 0.$

Now we get
\begin{eqnarray*}
& & \| S_n- \tilde S_n\|_\alpha \le \Gamma_5(\varepsilon)\|I + B_n\| \|U(\tau_n, \eta_{n})\|_\alpha  \\
& & + \varepsilon \| U_2(\eta_n, \tau_{n})\|_\alpha \|U(\tau_n, \eta_{n})\|_\alpha +
\Gamma_6(\varepsilon)\|U_2(\eta_{n+1}, \tilde \tau_n)\|_{\alpha}\|I + \tilde B_n\|  \le \Gamma_7(\varepsilon)
\end{eqnarray*}
and by (\ref{dd44})
\begin{eqnarray} \label{exp-2}
\| U(\eta_i, \eta_{i-k}) - U_2(\eta_i, \eta_{i-k})\|_\alpha \le M_1 e^{-\beta_1 \theta k} \Gamma_7(\varepsilon),
\end{eqnarray}
where $\Gamma_7(\varepsilon) \to 0$  as $\varepsilon \to 0.$

Continuing to evaluate $I_{15},$ we can obtain the inequalities
\begin{eqnarray*}
& & \|U_2(t,\eta_i)g\|_\alpha \le M_2 \|g\|_\alpha, \\
& & \|(U(t,\eta_i) - U_2(t,\eta_i))g\|_\alpha \le \Gamma_8(\varepsilon)\|g\|_\alpha, \\
& & \int_{\xi_{i-k}}^{\eta_{i-k+1}}\|(U(\eta_{i-k+1},s) - U_2(\eta_{i-k+1},s))f(s+h)\|_\alpha ds \le \Gamma_9(\varepsilon)\|f\|_{PC},
 \end{eqnarray*}
where $\Gamma_8(\varepsilon) \to 0$ and $\Gamma_9(\varepsilon) \to 0$ as $\varepsilon \to 0,$
$M_2$ is some positive constant. Note that as ealier, $t \in (\tau_i+ \varepsilon, \tau_{i+1}- \varepsilon).$

Taking into account the last inequalities, we conclude that series $I_{15}$ is convergent and
there exists $\Gamma_{10}(\varepsilon)$ such that
$I_{15} \le \Gamma_{10}(\varepsilon)\|f\|_{PC}$
and $\Gamma_{10}(\varepsilon) \to 0$  as $\varepsilon \to 0.$

Using estimations for $I_{11},..., I_{15}$, we get that there exists $\Gamma_{11}(\varepsilon)$ such that
\begin{eqnarray} \label{exp-3}
 \int_{-\infty}^{\infty} \|(G(t + r,s + r) - G(t,s))f(s+r)\|_\alpha ds \le \Gamma_{11}(\varepsilon)\|f\|_{PC}
\end{eqnarray}
and $\Gamma_{11}(\varepsilon) \to 0$  as $\varepsilon \to 0.$

By Lemma \ref{lem1}, $|\tau_{j+q} - \tau_j - h| < \varepsilon;$ therefore, $\tau_{j} + h +  \varepsilon > \tau_{j+q}$
(we assume that $h > 0$ for definiteness).
The difference $G(t, \tau_j+0) - G(t+h, \tau_{j+q}+0)$ is estimated as follows. Let $t - \tau_j \ge \varepsilon.$ Then
\begin{eqnarray}
& & \|( G(t, \tau_j+0) - G(t+h, \tau_{j+q}+0))g_{j+q} \|_\alpha  \nonumber \\
& & = \| (U(t, \tau_j+0)(I - P(\tau_j+0)) - U(t+h, \tau_{j+q}+0)(I - P(\tau_{j+q}+0)))g_{j+q} \|_\alpha  \nonumber \\
& & \le \|(U(t, \tau_j+0)(I - P(\tau_j+0)) - U(t, \tau_{j} + \varepsilon)(I - P(\tau_{j} + \varepsilon)))g_{j+q} \|_\alpha  \nonumber \\
& & + \| (U(t, \tau_{j} + \varepsilon)(I - P(\tau_{j} + \varepsilon)) - U(t + h, \tau_{j} +  \varepsilon + h)  \nonumber \\
& & \times (I - P(\tau_{j} + \varepsilon + h)))g_{j+q} \|_\alpha +
\| (U(t+h, \tau_{j+q}+0)(I - P(\tau_{j+q}+0))  \nonumber \\
& & - U(t + h, \tau_{j} + \varepsilon + h)(I - P(\tau_{j} + \varepsilon + h)))g_{j+q} \|_\alpha.   \label{ap2}
\end{eqnarray}
The first and third differences are small due to the continuity of function $U(t,s)$ at intervals between impulse points:
\begin{eqnarray*}
& & \|(U(t, \tau_j+0)(I - P(\tau_j+0)) - U(t, \tau_{j} + \varepsilon)(I - P(\tau_{j} + \varepsilon)))g_{j+q}\|_\alpha   \\
& & \le \| U(t, \tau_{j} + \varepsilon)(I - P(\tau_{j} + \varepsilon)) (U(\tau_{j} + \varepsilon, \tau_{j}+0)- I)g_{j+q}\|_\alpha  \\
& & \le \|(I - P(t))U(t,\tau_j +\varepsilon)\|_\alpha \|(U(\tau_{j} + \varepsilon, \tau_{j}+0)- I)g_{j+q}\|_\alpha  \\
& &  \le M e^{-\beta(t - \tau_j - \varepsilon)} \frac{1}{\alpha_1 - \alpha} C_{1-\alpha_1 + \alpha}\varepsilon^{\alpha_1-\alpha}\|g_{j+q}\|_{\alpha_1}, \\[2mm]
& & \| (U(t + h, \tau_{j} + \varepsilon + h)(I - P(\tau_{j} + \varepsilon + h))  \\
& & - U(t+h, \tau_{j+q}+0)(I - P(\tau_{j+q}+0)))g_{j+q} \|_\alpha  \\
& & =  \| U(t + h, \tau_{j} + \varepsilon + h)(I - P(\tau_{j} + \varepsilon + h))
(U(\tau_{j} + \varepsilon + h, \tau_{j+q}+0) - I)g_{j+q}\|_\alpha   \\
& & \le  M e^{-\beta(t - \tau_j - \varepsilon)}  \frac{1}{\alpha_1 - \alpha} C_{1-\alpha_1 + \alpha}\varepsilon^{\alpha_1-\alpha}\|g_{j+q}\|_{\alpha_1}.
\end{eqnarray*}

The second difference in (\ref{ap2}) is estimated using inequality (\ref{exp-2}) and the following transformation:
\begin{eqnarray*}
& & \| U(t, \tau_{j} + \varepsilon)(I - P(\tau_{j} + \varepsilon)) -  U(t + h, \tau_{j} + \varepsilon + h)(I - P(\tau_{j} +
\varepsilon + h)) \|_\alpha  \\
& & = \| U(t, \tau_{j} + \varepsilon)(I - P(\tau_{j} + \varepsilon)) -  U_2(t, \tau_{j} + \varepsilon)(I - P_2(\tau_{j} + \varepsilon)) \|_\alpha  \\
& & = \|U(t, \eta_{i})(I - P(\eta_{i}))U(\eta_{i}, \eta_{j+1})U(\eta_{j+1}, \tau_{j} + \varepsilon)  \\
& & - U_2(t, \eta_{i})(I - P(\eta_{i})U_2(\eta_{i}, \eta_{j+1})U_2(\eta_{j+1}, \tau_{j} + \varepsilon)\|_\alpha  \\
& & \le \| (U(t, \eta_{i}) - U_2(t, \eta_{i}))(I - P(\eta_{i}))U(\eta_{i}, \eta_{j+1})U(\eta_{j+1}, \tau_{j} + \varepsilon)\|_\alpha  \\
& & + \| U_1(t, \eta_{i})(P(\eta_{i})U(\eta_{i}, \eta_{j+1}) - P_2(\eta_{i})U_2(\eta_{i}, \eta_{j+1})U(\eta_{j+1}, \tau_{j} + \varepsilon)\|_\alpha  \\
& & + \|U_2(t, \eta_{i})P_2(\eta_{i})U_2(\eta_{i}, \eta_{j+1})(U(\eta_{j+1}, \tau_{j} + \varepsilon) - U_2(\eta_{j+1}, \tau_{j} + \varepsilon))\|_\alpha.
\end{eqnarray*}

Therefore,
\begin{eqnarray} \label{in7}
 \sum_{j \in {Z}} \| (G(t + h, \tau_{j+q}+0) - G(t,\tau_j+0))g_{j+q}\|_\alpha \le \Gamma_{12}(\varepsilon)\sup_j\|g_j\|_{\alpha_1},
\end{eqnarray}
where $\Gamma_{12}(\varepsilon) \to 0$  as $\varepsilon \to 0.$

The second integral and first sum in (\ref{in40}) are estimated as in (\ref{in30}):
\begin{eqnarray*}
\int_{-\infty}^\infty\|G(t,s))(f(s+h) - f(s))\|_\alpha ds + \sum_{j \in {Z}} \|U(t,\tau_j+0)(g_{j+q} - g_j)\|_\alpha \le
 M_3 \varepsilon
\end{eqnarray*}
since $h$ is $\varepsilon$-almost period of the right-hand side of the equation.
\vspace{2mm}

As a result of these evaluations, we get
$$\|u_0(t+h) - u_0(t)\|_\alpha \le \Gamma(\varepsilon) \ \ {\rm for} \ \ t \in {R}, \ |t-\tau_j| > \varepsilon, \ j \in {Z},$$
with $\Gamma(\varepsilon) \to 0$ as $\varepsilon \to 0.$ The last inequality implies that the function $u_0(t)$ is W-almost periodic
as function ${R} \to X^\alpha.$
\end{proof}

\begin{corollary}
Assume that Eq. (\ref{in1}), (\ref{in2}) satisfies the following:

i) conditions ${\bf (H1)}$ -- ${\bf (H3)}$, ${\bf (H7)};$

ii) the sequence $\{g_j\}$ of $g_j \in X^{\alpha}$ is almost periodic;

iii) the corresponding homogeneous equation is exponentially dichotomous.

Then the equation has a unique W-almost periodic solution $u_0(t) \in \mathcal{PC}({R}, X^\gamma)$ with $\gamma < \alpha.$
\end{corollary}

\vspace{2mm}

Now we consider a nonlinear equation with fixed moments of impulsive action:
\begin{eqnarray} \label{fix1}
& & \frac{du}{dt} + (A + A_1(t))u =  f(t, u), \quad t \not= \tau_j, \\
& & \Delta u|_{t=\tau_j} = u(\tau_j + 0) - u(\tau_j) = B_j u(\tau_j) + g_j(u(\tau_j)), \quad j \in Z, \label{fix2}
\end{eqnarray}

\begin{theorem} \label{theorem2}
Let ua consider Eq. (\ref{fix1}), (\ref{fix2}) in some domain
$U^\alpha_\rho = \{ x \in X^\alpha: \ \| x\|_\alpha \le \rho\}$ of space $X^\alpha.$
Assume that

1) the equation satisfies assumptions ${\bf (H1)}$ -- ${\bf (H4)}$, $\tau_j = \tau_j(0);$

2) the corresponding linear equation is exponentially dichotomous with constants $\beta > 0$ and $M \ge 1$;

3) the function $f(t,u): \ {R} \times U^\alpha_\rho \to X$ is continuous in $u$, W-almost periodic and locally H\"{o}lder continuous in $t$
uniformly with respect to $u \in U^\alpha_\rho$ and
there exist constants $N_1 > 0$ and  $\nu > 0$ such that
\begin{eqnarray*}
\|f(t_1, u_1) - f(t_2, u_2)\| \le N_1(|t_1 - t_2|^\nu + \|u_1 - u_2\|_\alpha)
\end{eqnarray*}
if $u_1, u_2 \in U_\rho^\alpha,$ and points $t_1$ and $t_2$ belong to the same interval of continuity;

4) the sequence $\{g_j(u)\}$ of continuous functions $U^\alpha_\rho \to X^{\alpha_1}$ is almost periodic uniformly with respect to $u \in U^\alpha_\rho$ and
$\|g_j(u_1) - g_j(u_2)\|_{\alpha} \le N_1 \|u_1 - u_2\|_\alpha, \ j \in {Z}.$
Also $\|g_j(u_1) - g_j(u_2)\|_{\alpha_1} \le N_1 \|u_1 - u_2\|_{\alpha_1}$ for $j \in {Z}$ and $u \in U^\alpha_\rho \cap X^{\alpha_1}$
with some $\alpha_1 > \alpha;$

5) the functions $\|f(t,0)\|_\alpha$ and $\|g_j(0)\|_{\alpha_1}$ are uniformly bounded for $t \in {R}, j \in {Z};$

 6) $N_1 M_* < 1$ and $\rho \ge M_0 M_* /(1 - N_1 M_*)$, where
 $$M_* = \frac{M_1}{1 - e^{-\beta_1 \theta}}\left( 1 + \frac{C_\alpha Q^{1-\alpha}}{1-\alpha}\right)$$
 and constants $\beta_1$ and $M_1$ are defined by Lemma \ref{lem4}.

Then in domain $U^\alpha_\rho$ for sufficiently small $N_1 > 0$ there exists a unique W-almost periodic solution $u_0(t)$
of Eq. (\ref{fix1}), (\ref{fix2}).

\end{theorem}

\begin{proof}
Denote by $\mathcal{M}_{\varrho}$ the set of all W-almost periodic functions $\varphi: {R} \to X^\alpha$
with discontinuity points $\tau_j, j \in {Z},$ satisfying the inequality $\|\varphi\|_{PC} \le \rho$.
In $\mathcal{M}_{\rho},$ we define the operator
\begin{eqnarray*}
(\mathcal{F}\varphi)(t) = \int_{-\infty}^{\infty} G(t,s)f(s,\varphi(s)) ds +
\sum_{j \in {Z}} G(t,\tau_j+0)g_j(\varphi(\tau_j)).
\end{eqnarray*}

Proceeding in the same way as in the proof of Theorem \ref{thm1}, we proof that
$(\mathcal{F}\varphi)(t)$ is a W-almost periodic function and
$\mathcal{F}: \mathcal{M}_{\rho} \to \mathcal{M}_{\rho}$ for $\rho > 0$ satisfying Condition 6).

Next, $\mathcal{F}$ is a contracting operator in $\mathcal{M}_{\rho}$ by sufficiently small $N_1 > 0.$

Hence, there exists $\varphi_0 \in \mathcal{M}_{\rho}$ such that
\begin{eqnarray*}
\varphi_0(t) = \int_{-\infty}^{\infty} G(t,s)f(s, \varphi_0(s)) ds +
\sum_{j \in {Z}} G(t,\tau_j+0)g_j(\varphi_0(\tau_j)).
\end{eqnarray*}

The function $\varphi_0(t)$ is locally H\"{o}lder continuous on every interval $(\tau_j, \tau_{j+1}), j \in {Z}$. Actually,
\begin{eqnarray*}
& & \varphi_0(t+\delta) - \varphi_0(t) = \int_{-\infty}^{\infty} G(t+\delta,s)f(s, \varphi_0(s)) ds -
\int_{-\infty}^{\infty} G(t,s)f(s, \varphi_0(s)) ds  \\
& & + \sum_{j \in {Z}} G(t + \delta,\tau_j+0)g_j(\varphi_0(\tau_j)) - \sum_{j \in {Z}} G(t,\tau_j+0)g_j(\varphi_0(\tau_j))  \\
& & = \int_{-\infty}^t(V(t+\delta, t) - I)U(t,s)(I-P(s))f(s, \varphi_0(s)) ds  \\
& & - \int^{\infty}_{t+\delta} (V(t+\delta, t) - I)U(t,s)P(s)f(s, \varphi_0(s)) ds  \\
& & + \int_{t}^{t+\delta} V(t+\delta, s)(I-P(s))f(s, \varphi_0(s)) ds +
 \int_{t}^{t+\delta} V(t, s)P(s)f(s, \varphi_0(s)) ds \\
& & + \sum_{\tau_j \le t}(V(t+\delta, t) - I) U(t,\tau_j+0)(I - P(\tau_j+0))g_j(\varphi_0(\tau_j))  \\
& & + \sum_{\tau_j > t}(V(t+\delta, t) - I) U(t,\tau_j+0) P(\tau_j+0)g_j(\varphi_0(\tau_j)).
\end{eqnarray*}
Applying (\ref{evop2}), (\ref{gr1}), (\ref{gr2}), and (\ref{in30}), we conclude that for every interval $t \in (t', t'')$
not containing impulse points $\tau_j,$ there exists positive constant $C$ such that
$\|\varphi_0(t+\delta) - \varphi_0(t)\|_\alpha \le  C \delta^{\alpha_1-\alpha}.$

The local H\"{o}lder continuity of $f(t, \varphi_0(t))$ follows from
\begin{eqnarray*}
& & \|f(t, \varphi_0(t)) - f(s, \varphi_0(s))\| \le N_1 \left( |t-s|^\nu + \|\varphi_0(t) - \varphi_0(s)\|_\alpha\right)  \\
 & & \le C_1\left(|t-s|^\nu + |t-s|^{\alpha_1-\alpha} \right).
\end{eqnarray*}
By Lemma 37, \cite{SP}, p. 214, if $\varphi_0(t)$ is W-almost periodic and $\inf_k (\tau_{k+1} - \tau_k) > 0$, then
$\{\varphi_0(\tau_k)\}$ is an almost periodic sequence.

The linear inhomogeneous equation
\begin{eqnarray} \label{fix1a}
& & \frac{du}{dt} + (A + A_1(t))u =  f(t, \varphi_0(t)), \quad t \not= \tau_j, \\
& & \Delta u|_{t=\tau_j} = u(\tau_j + 0) - u(\tau_j) = B_j u(\tau_j) + g_j(\varphi_0(\tau_j)), \quad j \in  Z, \label{fix2a}
\end{eqnarray}
has a unique W-almost periodic solution in the sense of Definition 4. Due to the uniqueness, it coincides with  $\varphi_0(t).$

Hence, the W-almost periodic function $\varphi_0(t): {R} \to X^\alpha$ satisfies Eq. (\ref{fix1})
for $t \in (\tau_j, \tau_{j+1})$ and difference equation (\ref{fix2}) for $t = \tau_j.$

\end{proof}

Now we study the stability of the almost periodic solution assuming exponential stability of the linear equation.
First, using ideas in \cite{RT}, we prove following generalized Gronwall inequality for impulsive systems.
\begin{lemma} \label{lem-in}
 Assume that $\{t_j\}$ is an increasing  sequence of real numbers such that $Q \ge t_{j+1} - t_{j} \ge \theta > 0$
 for all $j$,
 $M_1, M_2,$ and $M_3$ are positive constants, and $\alpha \in (0,1).$
 Then there exists a positive constant $\tilde C$ such that positive piecewise continuous function $u: [t_0,t] \to {R}$
 satisfying
\begin{eqnarray}
 & & z(t) \le M_1 z_0 + M_2 \sum_{j=1}^m \int_{t_{j-1}}^{t_j} (t_{j} - s)^{-\alpha}z(s)ds + M_2 \int_{t_{m}}^{t} (t - s)^{-\alpha}z(s)ds  \nonumber\\
& & + M_3 \sum_{j=1}^m z(t_j) \quad {\rm for} \quad t \in (t_m, t_{m+1}]  \label{gron1}
\end{eqnarray}
also satisfies
\begin{eqnarray} \label{gron2}
 z(t) \le M_1 z_0 \tilde C  \left(  1 + M_2 \tilde C \frac{Q^{1-\alpha}}{1-\alpha} + M_3 \tilde C\right)^{m}.
\end{eqnarray}
\end{lemma}

\begin{proof}
 We apply the method of mathematical induction. At the interval $t \in [t_0, t_1]$ the inequality (\ref{gron1}) has the form
 $$z(t) \le  M_1 z_0 + M_2 \int_{t_0}^{\tau_{1}} (\tau_{1} - s)^{-\alpha}z(s)ds.$$
 By Lemma \ref{lem-gr} there exists $\tilde C$ such that
 \begin{eqnarray} \label{gron3}
 0 \le z(t) \le M_1 z_0 \tilde C, \ \ t \in [t_0,t_1], \ \ \tilde C = \tilde C(M_1, M_2, Q).
 \end{eqnarray}
 Hence, (\ref{gron2}) is true for $t \in [t_0, t_1].$
 Assume (\ref{gron2}) is true for $t \in [t_0, t_n]$ and prove it for  $t \in (t_{n}, t_{n+1}].$
 Hence, for $t \in (t_{n}, t_{n+1}]$ we have
 \begin{eqnarray*}
  & &  z(t) \le M_1 z_0 + M_2 \int_{t_{0}}^{t_1} (t_{1} - s)^{-\alpha}z(s)ds + M_3 z(t_1) \\
& & + M_2 \sum_{j=2}^n \int_{t_{j-1}}^{t_j} (t_{j} - s)^{-\alpha}z(s)ds +
M_3 \sum_{j=1}^n z(t_j) + M_2 \int_{t_{n}}^{t} (t - s)^{-\alpha}z(s)ds   \\
& & \le M_1 z_0 + M_2 \frac{Q^{1-\alpha}}{1-\alpha}M_1 z_0 \tilde C + M_3 M_1 z_0 \tilde C + M_2 \int_{t_{n}}^{t} (t - s)^{-\alpha}z(s)ds \\
& & + \sum_{j=2}^n \left( 1 + M_2 \tilde C \frac{Q^{1-\alpha}}{1-\alpha} + M_3 \tilde C\right)^j
\left(  M_2 \tilde C \frac{Q^{1-\alpha}}{1-\alpha} + M_3 \tilde C \right)  M_1 z_0  \\
& & =  M_1 z_0 + M_2 \frac{Q^{1-\alpha}}{1-\alpha}M_1 z_0 \tilde C + M_3 M_1 z_0 \tilde C + M_2 \int_{t_{n}}^{t} (t - s)^{-\alpha}z(s)ds \\
& & + \sum_{j=2}^n \left( 1 + M_2 \tilde C \frac{Q^{1-\alpha}}{1-\alpha} + M_3 \tilde C\right)^{j-1}
\left[\left(1 + M_2 \tilde C \frac{Q^{1-\alpha}}{1-\alpha} + M_3 \tilde C \right) - 1\right] M_1 z_0  \\
& & \le M_1 z_0 \left( 1 + M_2 \frac{Q^{1-\alpha}}{1-\alpha} \tilde C + M_3 \tilde C\right)^{n} +  M_2 \int_{t_{n}}^{t} (t - s)^{-\alpha}z(s)ds.
 \end{eqnarray*}
 Hence, for $t \in (t_n, t_{n+1}],$ the function $z(t)$ satisfies the inequality
 $$z(t) \le C_1 +  M_2 \int_{t_{n}}^{t} (t - s)^{-\alpha}z(s)ds,$$
 where $C_1 =  M_1 z_0 \left( 1 + M_2 \frac{Q^{1-\alpha}}{1-\alpha} \tilde C + M_3 \tilde C\right)^{n}.$
 Applying (\ref{gron3}) at the interval $(t_n, t_{n+1}],$ we obtain (\ref{gron2}). The lemma is proved.
 \end{proof}

\begin{theorem} \label{theorem3}
Let Eq. (\ref{fix1}), (\ref{fix2}) satisfy assumptions of 
 Theorem \ref{theorem2} and let the corresponding linear equation be exponentially stable.

Then for sufficiently small $N_1 > 0,$ the equation has a unique W-almost periodic solution $u_0(t),$
and this solution is exponentially stable.
\end{theorem}

\begin{proof}
The existence and uniqueness of the W-almost periodic solution $u_0(t)$ follows from  Theorem \ref{theorem2}.
We prove its asymptotic stability. Let $u(t)$ be an arbitrary solution of the equation satisfying
$\|u(t_0) - u_0(t_0)\|_\alpha  \le \delta,$ where $\delta$ is small positive number.

Then by $t \ge t_0$ the difference of these solutions satisfies
 \begin{eqnarray*}
  & & u(t) - u_0(t) = U(t,t_0)( u(t_0) - u_0(t_0)) +  \int_{t_0}^t U(t,s)\Bigl(f(s,u(s)) - \\
  & & - f(s,u_0(s))\Bigl)ds + \sum_{t_0 \le \tau_k \le t} U(t,\tau_k + 0)\left( g_k(u(\tau_k)) - g_k(u_0(\tau_k)) \right).
 \end{eqnarray*}
Then for $t_0 \in (\tau_0, \tau_{1})$ and $t \in (\tau_j, \tau_{j+1}]$ we have
\begin{eqnarray*}
  & & \|u(t) - u_0(t)\|_\alpha \le \|U(t,t_0)(u(t_0) - u_0(t_0))\|_\alpha  \\
  & & + \int_{t_0}^{\tau_1}\|U(t,\tau_1)\|_\alpha \|V(\tau_1,s)(f(s,u(s)) - f(s,u_0(s)))\|_\alpha ds +... \\
  & & + \int_{\tau_{j-1}}^{\tau_j}\|U(t,\tau_j)\|_\alpha \|V(\tau_j,s)(f(s,u(s)) - f(s,u_0(s)))\|_\alpha ds \\
  & & + \int_{\tau_{j}}^{t} \|V(t,s)(f(s,u(s)) - f(s,u_0(s)))\|_\alpha ds + \\
  & & + \sum_{t_0 \le \tau_k \le t} \|U(t,\tau_k+0)\left( g_k(u(\tau_k)) - g_k(u_0(\tau_k)) \right)\|_\alpha  \\
  & & \le M e^{-\beta(t - t_0)}\|u(t_0) - u_0(t_0)\|_\alpha +
  M e^{-\beta(t - \tau_1)}\int_{t_0}^{\tau_1} \frac{L_Q N_1}{(\tau_1 - s)^\alpha}\|u(s) - u_0(s)\|_\alpha ds  \\
  & & +...+ M e^{-\beta(t - \tau_j)}\int_{\tau_{j-1}}^{\tau_j} \frac{L_Q N_1}{(\tau_j - s)^\alpha}\|u(s) - u_0(s)\|_\alpha ds \\
  & & + \int_{\tau_{j}}^{t} \frac{L_Q N_1}{(t - s)^\alpha}\|u(s) - u_0(s)\|_\alpha ds +
   \sum_{t_0 \le \tau_k \le t} M e^{-\beta(t - \tau_k)}N_1\|u(\tau_k) - u_0(\tau_k)\|_\alpha
 \end{eqnarray*}
 Denote $v(t) = e^{\beta t}\|u(t) - u_0(t)\|_\alpha, \ M_2 = e^{\beta Q}M L_Q N_1, M_3 = M N_1.$ Then
 \begin{eqnarray*}
 v(t) \le M v(t_0) + M_2 \int_{t_0}^{\tau_1} \frac{v(s)ds}{(\tau_1 - s)^\alpha} + ... +
 M_2 \int_{t_j}^{t} \frac{v(s)ds}{(\tau_j - s)^\alpha} + M_3 \sum_{k=1}^j v(\tau_k).
  \end{eqnarray*}
 Then by Lemma \ref{lem-in} we get
  \begin{eqnarray*}
 \|u(t) - u_0(t)\|_\alpha \le M \tilde C e^{-\beta(t-t_0)}\left( 1 + M_2 \tilde C \frac{Q^{1-\alpha}}{1-\alpha} + M_3 \tilde C\right)^{{i}(t,t_0)}
  \|u(t_0) - u_0(t_0)\|_\alpha
  \end{eqnarray*}
Therefore, if
  \begin{eqnarray*}
\beta > p \ln\left( 1 + M_2 \tilde C \frac{Q^{1-\alpha}}{1-\alpha} + M_3 \tilde C \right),
 \end{eqnarray*}
where $p$ is defined by (\ref{limp}), then W-almost periodic solution $u_0(t)$ of Eq. (\ref{fix1}), (\ref{fix2})
is asymptotically stable. This can be achieved by sufficiently small $N_1.$
\end{proof}


\section{Almost Periodic Solutions of Equations with Nonfixed Moments of Impulsive Action }
We consider the following equation with points of impulsive action
depending on solution
\begin{eqnarray} \label{ban1w}
& & \frac{du}{dt} + Au = f(t, u), \quad t \not= \tau_j(u), \\
& & u(\tau_j(u) + 0) - u(\tau_j(u)) = B_j u + g_j(u), \quad j \in  Z. \label{ban2w}
\end{eqnarray}

\begin{definition} \label{defst1} (\cite{LBS}).
 A solution $u_0(t)$ of Eq. (\ref{ban1w}), (\ref{ban2w}) defined for all $t \ge t_0,$ is called
 Lyapunov stable in space $X^\alpha$ if, for an arbitrary $\varepsilon > 0$ and $\eta > 0,$ there exists such a number
 $\delta = \delta(\varepsilon, \eta)$ that, for any other solution $u(t)$ of system,
 $\|u_0(t_0) - u(t_0)\|_\alpha < \delta$ implies that $\|u_0(t) - u(t)\|_\alpha < \varepsilon$
 for all $t \ge t_0$ such that $|t - \tau_j^0| > \eta,$ where $\tau_j^0$ are the times at which the solution
 $u_0(t)$ intersects the surfaces $t = \tau_j(u), j \in {Z}.$

 A solution $u_0(t)$ is said to be attractive, if for each $\varepsilon > 0, \eta > 0,$ and $t_0 \in {R},$
 there exist $\delta_0 = \delta_0(t_0)$ and $T = T(\delta_0, \varepsilon, \eta) > 0$ such that for any other solution $u(t)$
 of the system, $\|u_0(t_0) - u(t_0)\| < \delta_0$ implies $\|u_0(t) - u(t)\|_\alpha < \varepsilon$ for $t \ge t_0 + T$
 and $|t - \tau^0_k| > \eta.$

 A solution $u_0(t)$ is called asymptotically stable if it is stable and attractive.
 \end{definition}

\begin{theorem} \label{thm4}
Assume that in domain $U_\rho^\alpha = \{u \in X^\alpha, \|u\|_\alpha \le \rho\}$ Eq.
(\ref{ban1w}), (\ref{ban2w}) satisfies conditions (H1),  (H3) -- (H6), and

1) all solutions in domain $U_\rho^\alpha$ intersect each surface $t = \tau_j(u)$ no more then once;

2) $\| f(t_1, u) - f(t_2, u)\| \le H_1 |t_1 - t_2|,  \ H_1 > 0;$

3) $\| f(t, u_1) - f(t, u_2)\| + \|g_j(u_1) - g_j(u_2)\|_{\alpha} +
 |\tau_j(u_1) - \tau_j(u_2)| \le N_1\|u_1 - u_2\|_{\alpha},$ uniformly in $t \in {R}, u \in U_\rho^\alpha, j \in {Z}.$
 Also $\|g_j(u_1) - g_j(u_2)\|_{1} \le N_1 \|u_1 - u_2\|_{1}$ uniformly in $j \in Z, u \in U_\rho^\alpha \cap X^1;$

4) $A B_j = B_j A$ and $\|f(t,0)\| \le M_0, \ \|g_j(0)\|_{1} \le M_0$ for all $j \in {Z};$

5) the linear homogeneous equation
\begin{eqnarray} \label{lin1w}
& & \frac{du}{dt} + A u = 0, \quad t \not= \tau_j, \\
& & \Delta u|_{t=\tau_j} = u(\tau_j + 0) - u(\tau_j) = B_j u(\tau_j), \quad j \in Z, \label{lin2w}
\end{eqnarray}
 is exponentially stable in space $X^\alpha$
\begin{eqnarray*}
 \|U(t,s)u\|_\alpha \le M e^{-\beta (t-s)}\|u\|_\alpha, \ t\ge s, u \in X^\alpha,
\end{eqnarray*}
 where $\tau_j = \tau_j(0),$  $\beta > 0$ and $M \ge 1;$

 6) $N_1 M_* < 1$ and $\rho \ge M_0 M_* /(1 - N_1 M_*)$, where
 $$M_* = \frac{M_1}{1 - e^{-\beta_1 \theta}}\left( 1 + \frac{C_\alpha Q^{1-\alpha}}{1-\alpha}\right),$$
 where constants $\beta_1$ and $M_1$ are defined by Lemma \ref{lem4}.

Then for sufficiently small values of Lipschitz constant $N_1,$ Eq. (\ref{ban1w}), (\ref{ban2w})
has in $U_\rho^\alpha$ a unique W-almost periodic solution and this solution is exponentially stable.

\end{theorem}

\begin{proof}

1. First, using the method proposed in \cite{HPTT}, we proof the existence of the W-almost periodic solution.
 Let $y = \{y_j\}$ be an almost periodic sequence of elements $y_j \in X^\alpha, \|y_j\|_\alpha \le \rho.$
We consider the equation with fixed moments of impulsive action
\begin{eqnarray} \label{ban1y}
& & \frac{du}{dt} + A u = f(t, u), \quad t \not = \tau_j(y), \\
 & & u(\tau_j(y_j) + 0) - u(\tau_j(y_j)) = B_j u(\tau_j(y_j)) + g_j(y_j), \quad j \in Z. \label{ban2y}
\end{eqnarray}

By Lemma \ref{lem4}, if a constant $N_1$ sufficiently small, then corresponding to (\ref{ban1y}) and (\ref{ban2y}) the linear impulsive
equation (if $f \equiv 0, \ g_j(y_j) \equiv 0, j \in {Z},$) is exponentially stable.
Its evolution operator $U(t,\tau, y)$ satisfies estimate
\begin{eqnarray*}
\|U(t,\tau, y)u\|_\alpha \le M_1 e^{-\beta_1(t - \tau)}\|u\|_\alpha, \ t \ge \tau,
\end{eqnarray*}
with some positive constants $M_1 \ge M, \beta_1 \le \beta.$

The equation (\ref{ban1y}), (\ref{ban2y}) has a unique solution bounded on the axis which satisfies the integral equation
 \begin{eqnarray} \label{ban3y}
 u(t,y) = \int_{-\infty}^t U(t,\tau,y) f(\tau, u(\tau,y))d\tau + \sum_{\tau_j(y_j) \le t}
U(t,\tau_j(y_j)+0,y)g_j(y_j).
\end{eqnarray}

We choose $u_0(t,y) \equiv 0$ and construct the sequence of W-almost periodic functions \ $u_{n+1}(t,y) =$
$$ = \int_{-\infty}^t U(t,\tau,y) f(\tau, u_n(\tau,y))d\tau + \sum_{\tau_j(y_j) \le t}
U(t,\tau_j(y_j)+0,y)g_j(y_j), \ n = 0,1,... .$$
The proof of the W-almost periodicity of $u_{n+1}(t,y)$ in space $X^\alpha$ is similar to the proof of Theorem \ref{thm1}.

One can verify that for sufficiently small $N_1 > 0$ the sequence $\{u_{n}(t,y)\}$ converges to the
W-almost periodic solution $u^*(t,y): {R} \to X^\alpha$ of Eq. (\ref{ban3y}).
As in the proof to Theorem \ref{theorem2}, we prove that $u^*(t,y)$ is the W-almost periodic solution
of impulsive equation (\ref{ban1y}), (\ref{ban2y}).

Let $t \in (\tilde\tau_i, \tilde\tau_{i+1}),$ where $\tilde\tau_i = \tau_i(y_i).$
As in (\ref{in30}), we obtain
\begin{eqnarray*}
& & \|u^*(t,y)\|_\alpha \le \int_{-\infty}^t\|A^\alpha U(t,s, y)(f(s, 0) + f(s, u^*(s,y)) - f(s, 0))\| ds  \nonumber\\
& & + \sum_{\tau_j(y_j) \le t} \| U(t, \tilde\tau_j+0, y)(g_j(0) + g_j(y_j) - g_j(0))\|_\alpha  \nonumber\\
& & \le \frac{M_1}{1 - e^{-\beta_1 \theta}}\left( \frac{C_\alpha \Theta^{1-\alpha}}{1-\alpha}\left(M_0 + N_1 \sup_t\|u^*(t,y)\|_\alpha \right)
+ M_0 + N_1 \sup_j\|y_j\|_\alpha \right).
\end{eqnarray*}
Hence, $\sup_t\|u^*(t,y)\| \le \rho,$ where $\rho$ satisfies Condition 6).

If we choose the almost periodic sequence $y^* = \{y_j^*\}, y_j^* \in X^\alpha,$ such that
$$u^*(\tau_j( y_j^*), y^*) =  y_j^*$$
for all $j \in {Z},$ then the function $u^*(t, y^*)$ will be exactly
the W-almost period solution of Eq. (\ref{ban1w}), (\ref{ban2w}).

We consider the space $\mathcal{N}$ of sequences $y = \{y_j\}, \ y_j \in X^\alpha,$ with norm
$\| y\|_S = \sup_{j}\|y_j\|_\alpha$ and map $S: \ \mathcal{N} \to \mathcal{N},$
$$S(y) = \{u^*(\tau_j(y_j), y)\}_{j \in {Z}}.$$

$S$ maps the domain $\mathcal{N}_\rho = \{y \in \mathcal{N}, \|y\|_S \le \rho\}$ into itself.

Now we proof that $S$ is a contraction. Let, for definiteness, $\tilde\tau_j^1 = \tau_j(y_j) < \tilde\tau_j^2 = \tau_j(z_j).$ Then
 \begin{eqnarray}
& & \|S(y)_j - S(z)_j\|_\alpha =
\| u^*(\tau_j(y_j), y) -
u^*(\tau_j(z_j), z)\|_\alpha  \nonumber\\
& & \le \| u^*(\tilde\tau_j^1, y) - u^*(\tilde\tau_j^1, z)\|_\alpha +
\|u^*(\tilde\tau_j^1, z) - u^*(\tilde\tau_j^2, z)\|_\alpha, \label{rizn1}
\end{eqnarray}

\bigskip

 Denote $$\mathcal{J} = \cup_j \mathcal{J}_j, \ \
\mathcal{J}_j = (\max\{\tilde\tau_{j-1}^1,\tilde\tau_{j-1}^2\}, \min\{\tilde\tau_j^1,\tilde\tau_j^2\}] = (\tau''_{j-1},\tau'_j].$$

Denote also $\xi_j = (\tau'_j + \tau''_{j-1})/2, \ j \in Z.$

To estimate the difference $\|u^*(\tilde\tau_j^1, y) - u^*(\tilde\tau_j^1, z)\|_\alpha,$ we apply iteration on $n.$
Put $u_{0}(t,y) = u_{0}(t,z) = 0.$ Then for $t \in (\tau''_i, \tau'_{i+1}]$ we get
\begin{eqnarray}
& & \|u_{1}(t,y) - u_{1}(t,z)\|_\alpha
 =  \| \sum_{k \le i}  A^{\alpha}U(t,\tilde\tau_k^1+0,y)g_k(y_k) -
\sum_{k \le i}  A^{\alpha}U(t,\tilde\tau_k^2+0,z)g_k(z_k) \|  \nonumber\\
& & + I_0 \le I_0 + \sum_{k \le i} \| A^{\alpha} U(t,\tilde\tau_k^1+0,y)\left( g_k(y_k) -  g_k(z_k)\right)\| +
\| A^{\alpha}\Bigl( U(t,\tilde\tau_i^1+0,y)  \nonumber\\
& & -  U(t,\tilde\tau_i^2+0,z)\Bigl) g_i(z_i)\|
 + \sum_{k < i} \| A^{\alpha}\left( U(t,\tilde\tau_k^1+0,y) -  U(t,\tilde\tau_k^2+0,z)\right) g_k(z_k)\|  \nonumber\\
& & \le I_0 + \sum_{k \le i} M_1 e^{-\beta_1|t - \tilde\tau_k^1|}N_1\|y_k-z_k\|_\alpha +
\|A^\alpha e^{-A(t - \tau''_i)}(e^{-A(\tau''_i -\tau'_i)} - I)g_i(z_i)\|  \nonumber \\
& & + \sum_{k < i} \|(U(t,\xi_i,y)(U(\xi_i, \xi_{k+1},y) U(\xi_{k+1}, \tilde \tau_k^1+0,y)  \nonumber \\
& & - U(t,\xi_i,z)(U(\xi_i, \xi_{k+1},z) U(\xi_{k+1}, \tilde \tau_k^2+0,z))g_k(z_k)\|_\alpha  \nonumber\\
& & \le I_0 +  \frac{M_1 N_1}{1 - e^{-\beta_1 \theta}}\|y-z\|_S +  C_0  C_{1} (t - \tau''_i)^{-\alpha}
|\tau''_i - \tau'_i|\|g_i(z_i)\|_{1}  \nonumber \\
& & + \sum_{k < i} \|A^\alpha (U(t,\xi_i,y) - U(t,\xi_i,z))U(\xi_i, \xi_{k+1},y) U(\xi_{k+1}, \tilde \tau_k^1+0,y)g_k(z_k)\|  \nonumber\\
& & + \sum_{k < i} \|A^\alpha U(t,\xi_i,z)(U(\xi_i, \xi_{k+1},y) - U(\xi_i, \xi_{k+1},z))U(\xi_k, \tilde \tau_k^1+0,y)g_k(z_k)\|  \nonumber\\
& & + \sum_{k < i} \|A^\alpha U(t,\xi_{k+1},z) (U(\xi_{k+1}, \tilde \tau_k^1+0,y) - U(\xi_{k+1}, \tilde \tau_k^2+0,z))g_k(z_k)\|,
 \label{rizn1a}
\end{eqnarray}
where
$$I_0 = \int_{-\infty}^t \| A^\alpha ( U(t,s,y) - U(t,s,z))f(s,0)\| ds.$$

To evaluate the difference $U(\xi_i, \xi_{k+1},y) - U(\xi_i, \xi_{k+1},z)$ we construct two
sequences of bounded operators $X^\alpha \to X^\alpha$ defined by
$$T_n = U(\xi_{n+1}, \xi_n,y), \ \tilde T_n = U(\xi_{n+1}, \xi_n,z), \ n \in {Z}.$$
The corresponding difference equations
$u_{n+1} = T_n u_n$ and $u_{n+1} = \tilde T_n u_n$
are exponentially stable. Their evolution operators
$$T_{n,m} = T_{n-1}...T_m, \ n \ge m, \ T_{m,m}= I,$$
and
$$\tilde T_{n,m} = \tilde T_{n-1}...\tilde T_m, \ n \ge m, \ \tilde T_{m,m}= I,$$
 satisfies equality
\begin{eqnarray*} 
\tilde T_{n,m} - T_{n,m} = \sum_{k < n} T_{n,k+1}(\tilde T_k - T_k) \tilde T_{k,m}, \ n \ge m.
\end{eqnarray*}

Analogous to (\ref{dd3}) and (\ref{dd4}), we obtain
\begin{eqnarray} \label{exp-2d}
\|\tilde T_{n,m} - T_{n,m}\|_{\alpha} \le M_2 e^{-\beta_2 \theta(n - m)} \sup_k \|\tilde T_k - T_k\|_\alpha, \ n \ge m,
\end{eqnarray}
with some $\beta_2 \le \beta_1, M_2 \ge M_1.$

Now we estimate the difference $\|\tilde T_n - T_n\|_\alpha:$
\begin{eqnarray*}
& & \| T_n - \tilde T_n\|_\alpha  = \|U(\xi_{n+1},\xi_n,y) - U(\xi_{n+1},\xi_n,z)\|_\alpha    \nonumber\\
& & = \|e^{-A(\xi_{n+1} - \tilde\tau_n^1)}(I + B_n)e^{-A(\tilde\tau_n^1 - \xi_{n})} -
 e^{-A(\xi_{n+1} - \tilde\tau_n^2)}(I + B_n)e^{-A(\tilde\tau_n^2 - \xi_{n})}\|_\alpha   \nonumber\\\
 & & \le \|A^\alpha(e^{-A(\xi_{n+1} - \tilde\tau_n^1)} - e^{-A(\xi_{n+1} - \tilde\tau_n^2)})(I + B_n)e^{-A(\tilde\tau_n^1 - \xi_{n})}\|  \nonumber\\
& & + \| A^\alpha e^{-A(\xi_{n+1} - \tilde\tau_n^2)}(I + B_n)(e^{-A(\tilde\tau_n^1 - \xi_{n})} - e^{-A(\tilde\tau_n^2 - \xi_{n})}) \|  \nonumber\\
& & \le 2C_\alpha C_1 (\theta/2)^{-1-\alpha}(1 + b) |\tilde\tau_n^1 - \tilde\tau_n^2|. 
\end{eqnarray*}
Therefore,
\begin{eqnarray} \label{exp-dd}
& & \|(\tilde T_{n,m} - T_{n,m})u\|_{\alpha}  = \| (U(\xi_n,\xi_m,y) - U(\xi_n,\xi_m,z))u\|_\alpha  \nonumber\\
& & \le 2 M_2 e^{-\beta_2 \theta(n - m)}  C_\alpha C_1 (\theta/2)^{-1-\alpha}(1+b) \sup_j |\tilde\tau_i^1 - \tilde\tau_i^2| \|u\|_\alpha, \ n \ge m.
\end{eqnarray}

To finish the estimation of  (\ref{rizn1a}), we consider the following two differences
\begin{eqnarray}
 & & \|( U(t, \xi_i,y) - U(t, \xi_i,z)))u\|_\alpha \le \| A^\alpha(e^{-A(t - \tau'_i)}(I+ B_i) e^{-A(\tau'_i - \xi_i)}  \nonumber\\
 & &  -  e^{-A(t - \tau''_i)}(I+ B_i) e^{-A(\tau''_i - \xi_i)})u\| \le
 \frac{4 C_0 C_\alpha C_{1-\alpha}(1+b)}{\theta(t - \tau''_i)^\alpha}|\tau''_i - \tau'_i| \|u\|_\alpha. \label{exp-dd0} \\
& & \|(U(\xi_k, \tilde \tau_k^1+0,y) - U(\xi_k, \tilde \tau_k^2+0,z))u\|_\alpha =
 \|A^\alpha(I - e^{-A( \tau''_k - \tau'_k)})e^{-A(\xi_{k+1} - \tau''_k)}u\| \nonumber\\
 & & \le C_0 C_1 (2/\theta)|\tau''_i - \tau'_i|\|u\|_\alpha. \label{exp-ddd}
\end{eqnarray}
Taking into account (\ref{exp-2d}), (\ref{exp-dd0}) and (\ref{exp-ddd}), by (\ref{rizn1a}) we obtain
for $t \in (\tau''_i, \tau'_{i+1}]$
\begin{eqnarray}
 \|u_{1}(t,y) - u_{1}(t,z)\|_\alpha
 \le N_1 \|y-z\|_S \left(K'_1 + K''_1 (t-\tau''_i)^{-\alpha}\right) + I_0,
 \label{rizn1aa}
\end{eqnarray}
where the positive constants $K'_1$ and $K''_1$ don't depend on $i.$

Now we consider the $(n+1)$st iteration
\begin{eqnarray}
& & \|u_{n+1}(t,y) - u_{n+1}(t,z)\|_\alpha  \nonumber\\
& & =  \| \int_{-\infty}^t A^{\alpha}U(t,\tau,y) f(\tau, u_n(\tau,y))d\tau +
\sum_{k \le i}  A^{\alpha}U(t,\tilde\tau_k^1+0,y)g_k(y_k)  \nonumber\\
& & - \int_{-\infty}^t A^{\alpha}U(t,\tau,z) f(\tau, u_n(\tau,z))d\tau -
\sum_{k \le i}  A^{\alpha}U(t,\tilde\tau_k^2+0,z)g_k(z_k) \|  \nonumber\\
 & & \le \int_{-\infty}^{t} \| A^\alpha U(t,\tau,y)\left( f(\tau, u_n(\tau,y)) -
f(\tau, u_n(\tau,z))\right)\| d\tau  \nonumber\\
& & + \int_{-\infty}^{t}\| A^\alpha (U(t,\tau,y) - U(t,\tau,z))f(\tau, u_n(\tau,z))\|d\tau  \nonumber\\
& & + \sum_{k \le i} \| A^{\alpha} U(t,\tilde\tau_k^1+0,y)\left( g_k(y_k) - g_k(z_k)\right)\|  \nonumber\\
& & + \sum_{k \le i} \| A^{\alpha}\left(  U(t,\tilde\tau_k^1+0,y) -  U(t,\tilde\tau_k^2+0,z)\right) g_k(z_k)\|. \label{in4}
\end{eqnarray}

Similar to (\ref{in30}), we get
\begin{eqnarray*}
& & \int_{\tau''_i}^t \|A^\alpha e^{-A(t-s)}\left( f(\tau, u_n(\tau,y)) -
f(\tau, u_n(\tau,z))\right)\| d\tau  \nonumber\\
& &  + \sum_{k < i} \int_{\tau''_k}^{\tau'_{k+1}} \| A^\alpha U(t,\tau,y)\left( f(\tau, u_n(\tau,y)) -
f(\tau, u_n(\tau,z))\right)\| d\tau  \nonumber\\
& & \le \frac{M_1}{1 - e^{-\theta\beta_1}} \frac{C_\alpha \Theta^{1-\alpha}}{1-\alpha}
N_1 \sup_{\tau \in \mathcal{J}} \| u_n(\tau,y) -  u_n(\tau,z)\|, \nonumber\\[2mm]
& &  \sum_{k \le i} \| A^{\alpha} U(t,\tilde\tau_k^1+0,y)\left( g_k(y_k) - g_k(z_k)\right)\| \le
 \frac{M_1}{1 - e^{-\theta\beta_1}} N_1 \|y-z\|_\alpha.
 \end{eqnarray*}

 If $\|u_n(\tau,y)\|_\alpha \le \rho$ and $\|u_n(\tau,z)\|_\alpha \le \rho,$ then for $t \in (\tau''_i, \tau'_{i+1}]$
\begin{eqnarray}
& & \sum_{k \le i} \int_{\tau'_k}^{\tau''_k} \| A^\alpha U(t,s, y)\left( f(s, u_n(s,y)) -
f(s, u_n(s,z))\right)\| ds  \nonumber\\
& & \le \sum_{k \le i} \int_{\tau'_k}^{\tau''_{k}} \|  U(t,s, y) f(s, u_n(s,y))\|_\alpha ds +
\sum_{k \le i} \int_{\tau'_k}^{\tau''_{k}} \|  U(t,s,y) f(s, u_n(s,z))\|_\alpha ds  \nonumber\\
& & \le 2\sum_{k < i} M_1 e^{-\beta_1 |t -  \tau''_{k}|}(M_0 + N_1\rho) +
2\int_{\tau'_i}^{\tau''_{i}} \| A^\alpha U(t,s,y)\|(M_0 + N_1\rho) ds   \nonumber\\
& & \le \left(\frac{2 M_1}{1 - e^{-\beta_1 \theta}} + \frac{2 M_1}{1-\alpha}(t - \tau''_i)^{-\alpha}\right)(M_0 + N_1\rho)N_1 \|y-z\|_S,
\label{rizn1b}
\end{eqnarray}
since for $t > \tau_2 > \tau_1$
$$\int_{\tau_1}^{\tau_2}\frac{ds}{(t-s)^\alpha} \le \frac{\tau_2 - \tau_1}{(1-\alpha)((t-\tau_2)^\alpha}.$$

The second integral in (\ref{in4}) satisfies the following inequality:
\begin{eqnarray}
& & I_2 = \int_{-\infty}^{t}\| A^\alpha (U(t,s,y) - U(t,s,z))f(s, u_n(s,z))\|ds
 \le \int_{\tau''_i}^t \|A^\alpha (e^{-A(t-s)} \nonumber\\
 & & - e^{-A(t-s)})f(s, u_n(s,z))\| ds +
 \int_{\tau'_i}^{\tau''_i} \|A^\alpha(U(t,s,y) - U(t,s,z))f(s, u_n(s,z))\| ds  \nonumber\\
& & + \int_{\xi_i}^{\tau'_i} \| A^\alpha (U(t,s,y) - U(t,s,z))f(s, u_n(s,z))\| ds  \nonumber\\
& & + \sum_{k< i} \int_{\xi_{k}}^{\xi_{k+1}} \| A^\alpha (U(t,s,y) - U(t,s,z))f(s, u_n(s,z))\| ds. \label{in5}
\end{eqnarray}

We consider all integrals in (\ref{in5}) separately.
\begin{eqnarray*}
& & I_{21} = \int_{\tau'_i}^{\tau''_i} \|A^\alpha U(t,s,y)f(s, u_n(s,z))\| ds \le
\frac{C_\alpha (1+b) (M_0 + N_1 \rho)}
 {(1-\alpha)(t - \tau''_i)^{\alpha}} |\tau''_i - \tau'_i|, \nonumber\\
& & I_{22} = \int_{\tau'_i}^{\tau''_i} \|A^\alpha U(t,s,z)f(s, u_n(s,z))\| ds
 \le \frac{C_\alpha (1+b) (M_0 + N_1 \rho)}
 {(1-\alpha)(t - \tau''_i)^{\alpha}} |\tau''_i - \tau'_i|, \\
& & I_{23} = \int_{\xi_i}^{\tau'_i} \| A^\alpha (U(t,s,y) - U(t,s,z))f(s, u_n(s,z))\| ds  \nonumber\\
& & = \int_{\xi_i}^{\tau'_i} \| A^\alpha (U(t,\tilde\tau_i^1,y)U(\tilde\tau_i^1,s,y) -
 U(t,\tilde\tau_i^2,z)U(\tilde\tau_i^2,s,z))f(s, u_n(s,z))\| ds   \nonumber\\
& & \le \int_{\xi_i}^{\tau'_i} \| A^\alpha \Bigl((e^{-A(t - \tilde\tau_i^1)} - e^{-A(t - \tilde\tau_i^2)})(I+B_i)
e^{-A(\tilde\tau_i^1 - s)}  \nonumber\\
& & - A^\alpha e^{-A(t - \tilde\tau_i^2)})(I+B_i)(e^{-A(\tilde\tau_i^1 - s)} - e^{-A(\tilde\tau_i^2 - s)})\Bigl)f(s, u_n(s,z))\| ds   \nonumber\\
& & \le \frac{2 C_0 C_{\alpha_1} C_{1+\alpha-\alpha_1}}{(t -\tau''_i)^{\alpha_1}}(1+b) \frac{(\tau'_i - \xi_i)^{\alpha_1-\alpha}}{\alpha_1-\alpha}
|\tau''_i - \tau'_i|, \quad \alpha_1 > \alpha.
\end{eqnarray*}

The last sum in (\ref{in5}) is transformed as follows:
\begin{eqnarray*}
& & I_{24} = \sum_{k < i} \int_{\xi_{k}}^{\xi_{k+1}} \| A^\alpha (U(t,s,y) - U(t,s,z))f(s, u_n(s,z))\| ds  \\
& & = \sum_{k < i} \int_{\xi_{k}}^{\xi_{k+1}} \|(U(t, \xi_i,y)U(\xi_i, \xi_{k+1},y)U(\xi_{k+1},s,y)  \\
& & - U(t, \xi_i,z) U(\xi_i, \xi_{k+1},z) U(\xi_{k+1},s,z))f(s, u_n(s,z))\|_{\alpha} ds  \\
& & \le \sum_{k < i} \int_{\xi_{k}}^{\xi_{k+1}} \Biggl(\|(  U(t, \xi_i,y) -
U(t, \xi_i,z)) U(\xi_i, \xi_{k+1},y)U(\xi_{k+1},s,y) f(s, u_n(s,z))\|_\alpha  \\
& & + \|U(t, \xi_i,z)( U(\xi_i, \xi_{k+1},y) - U(\xi_i, \xi_{k+1},z))
U(\xi_{k+1},s,y) f(s, u_n(s,z))\|_\alpha  \\
& & + \|U(t, \xi_i,z)U(\xi_i, \xi_{k+1},z)
(U(\xi_{k+1},s,y) - U(\xi_{k+1},s,z))f(s, u_n(s,z))\|_\alpha \Biggl)ds.
\end{eqnarray*}

To finish the estimation of integral $I_{24}$ we use (\ref{exp-dd}), (\ref{exp-dd0}), (\ref{exp-ddd}) and
\begin{eqnarray*}
& &  \int_{\xi_{k}}^{\xi_{k+1}} \|A^\alpha (U(\xi_{k+1},s,y) - U(\xi_{k+1},s,z))f\|_\alpha ds  \\
& & \le \int_{\tau''_{k}}^{\xi_{k+1}}  \|A^\alpha ( e^{-A(\xi_{k+1} - s)} - e^{-A(\xi_{k+1} - s)})f\| ds  \\
& & + \int_{\tau'_{k}}^{\tau''_{k}}\|A^\alpha( e^{-A(\xi_{k+1} - \tau''_{k})}(I + B_{k})
e^{-A(\tau''_{k} - s)} - e^{-A(\xi_{k+1} - s)})f\| ds  \\
& & + \int_{\xi_{k}}^{\tau'_{k}} \| ( e^{-A(\xi_{k+1} - \tau'_{k})}(I + B_{k})
e^{-A(\tau'_{k} - s)} -
 e^{-A(\xi_{k+1} - \tau''_{k})}(I + B_{k}) e^{-A(\tau''_{k} - s)}) f\|_\alpha ds  \\
& & \le \tilde K (\xi_{k+1} - \tau''_k)^{-\alpha_1} (1+b)|\tau''_{k} -\tau'_{k}| \|f\|
\end{eqnarray*}
with some positive constant $\tilde K$ and $\alpha_1 > \alpha.$ Therefore,
\begin{eqnarray} \label{rizn2}
 I_{2} \le N_1 \left( K'_2 + \frac{ K''_2}{(t- \tau''_i)^{\alpha_1}}\right)\|y-z\|_S
\end{eqnarray}
with $\alpha_1 > \alpha$ and positive constants $K'_2$ and $K''_2$ independent on $i, k.$

By (\ref{rizn1aa}), (\ref{in5}), and (\ref{rizn2}) we obtain for $t \in (\tau''_i, \tau'_{i+1}]$
\begin{eqnarray}
& & \|u_{n+1}(t,y) - u_{n+1}(t,z)\|_\alpha \le
 \sum_{k < i} \int_{\tau''_k}^{\tau'_{k+1}} \| A^\alpha U(t,\tau,y)( f(\tau, u_n(\tau,y))  \nonumber\\
& & -f(\tau, u_n(\tau,z)))\| d\tau +
\int_{\tau''_i}^{t}\| A^\alpha U(t,\tau,y)( f(\tau, u_n(\tau,y))  \nonumber\\
& & - f(\tau, u_n(\tau,z)))\| d\tau +
\left(K'_3 + \frac{K''_3}{(t - \tau''_i)^{\alpha_1}}\right)N_1 \|y-z\|_S,
 \label{rizn2a}
\end{eqnarray}
where the constants $K'_3$ and $K''_3$ don't depend on $n.$

Let the $n$th iteration satisfies the inequality
$$\|u_{n}(t,y) - u_{n}(t,z)\|_\alpha \le \left(L'_n + \frac{L''_n}{(t -  \tau''_i)^{\alpha_1}}\right)N_1 \|y-z\|_S,
\ t \in (\tau''_i, \tau'_{i+1}],$$
with positive constants $L'_n$ and $L''_n.$ We estimate the $(n+1)$st iteration.
\begin{eqnarray}
& & \|u_{n+1}(t,y) - u_{n+1}(t,z)\|_\alpha \le \left(K'_3 + \frac{K''_3}{(t -  \tau''_i)^{\alpha_1}}\right)N_1 \|y-z\|_S  \nonumber\\
& & + N_1^2\|y-z\|_S \sum_{k < i} \int_{\tau''_k}^{\tau'_{k+1}}\|A^\alpha U(t,s)\| \left( L'_n + \frac{L''_n}{(s - \tau''_k)^{\alpha_1}}\right)ds
\nonumber\\
& & + N_1^2\|y-z\|_S \int_{\tau''_i}^{t}\|A^\alpha U(t,s)\| \left( L'_n + \frac{L''_n}{(s - \tau''_i)^{\alpha_1}}
\right)ds  \nonumber\\
& & \le  N_1^2\|y-z\|_S \Biggl(\sum_{k < i} \int_{\tau''_k}^{\tau'_{k+1}}M_1 e^{-\beta_1|t-s|} \left( L'_n +
\frac{L''_n}{(s - \tau''_k)^{\alpha_1}}\right)ds  \nonumber\\
& & + \int_{\tau''_{i}}^t \frac{M_1}{(t-s)^{\alpha_1}} \left( L'_n + \frac{L''_n}{(s - \tau''_i)^{\alpha_1}}\right)ds \Biggl)
 + \left(K'_3 + \frac{K''_3}{(t - \tau''_i)^{\alpha_1}}\right)N_1 \|y-z\|_S \nonumber\\[2mm]
& & \le \Biggl( \frac{M_1}{1 - e^{-\beta_1\theta}}\left( L'_n Q + \frac{L''_n Q^{1-\alpha_1}}{1-\alpha_1}\right) +
\frac{L''_n M_1 2^{2\alpha}}{1-\alpha_1}(t - \tau''_i)^{1-2\alpha_1}  \nonumber\\
& &  +\frac{L'_n M_1}{1-\alpha_1}(t - \tau''_i)^{1-\alpha_1}\Biggl)N_1^2\|y-z\|_S  +
 \left(K'_3 + \frac{K''_3}{(t - \tau''_i)^{\alpha_1}}\right)N_1 \|y-z\|_S  \nonumber\\
 & & = \left(L'_{n+1} + \frac{L''_{n+1}}{(t - \tau''_i)^{\alpha_1}}\right)N_1 \|y-z\|_S. \label{rizn3}
\end{eqnarray}

One can verify that for sufficiently small $N_1$ the sequences $L'_{n}$ and $L''_{n}$ are uniformly bounded by some constants
$L'_{*}$ and $L''_{*}.$

Since the sequences $u_{n}(t,y)$ and $u_{n}(t,z)$ tend to limit functions $u_*(t,y)$  and $u_*(t,z)$, respectively,
we conclude by (\ref{rizn3}) for $t \in (\tau''_i, \tau'_{i+1}]$ that
\begin{eqnarray*}
 \|u_{*}(t,y) - u_{*}(t,z)\|_\alpha \le \left(L'_{*} + \frac{L''_{*}}{(t - \tau''_i)^{\alpha_1}}\right)N_1 \|y-z\|_S
\end{eqnarray*}
and
\begin{eqnarray}
\|u^*(\tau'_{i+1}, y) - u^*(\tau'_{i+1}, z)\|_\alpha \le \left( L'_{*} + \frac{L''_{*}}{\theta^{\alpha_1}}\right)N_1\|y-z\|_S.
\label{rizn9}
\end{eqnarray}

\vspace{2mm}

Now we estimate the second summand in (\ref{rizn1}). Note that by our assumption
$\tilde\tau_j^1 < \tilde\tau_j^2.$
\begin{eqnarray*}
\|u^*(\tilde\tau_j^1, z) - u^*(\tilde\tau_j^2, z)\|_\alpha =
\Bigl\|\int_{\tilde\tau_j^1}^{\tilde\tau_j^2} \frac{d}{ds}u^*(s, z) ds \Bigl\|_\alpha.
\end{eqnarray*}

By Theorem 3.5.2, \cite{H}, at the interval $(\tau''_{j-1},\tau'_j)$ the derivative satisfies
$$\Bigl\|\frac{d}{ds}u^*(s, z)\Bigl\|_\gamma \le \tilde K_1 (s - \tau''_{j-1})^{\alpha-\gamma-1}$$
with some positive const $\tilde K_1 $ independent of $j$ and an initial value from $U_\rho^\alpha.$

Then for $t \in (\tilde\tau_j^1,\tilde\tau_j^2)$
$$\Bigl \|\frac{d}{ds}u^*(s, z)\Bigl\|_\gamma \le \tilde K_1 \left(\frac{\theta}{2}\right)^{\alpha-\gamma-1} = \tilde K_2$$
and
\begin{eqnarray}
\|u^*(\tilde\tau_j^1, z) - u^*(\tilde\tau_j^2, z)\|_\alpha \le  \tilde K_2 |\tilde\tau_{j}^1 - \tilde\tau_j^2| \le
\tilde K_2 N_1 \|y-z\|_S.
\label{rizn10}
\end{eqnarray}

By (\ref{rizn9}) and (\ref{rizn10}) we have
\begin{eqnarray}
& & \|u^*(\tilde\tau_j^1, z) - u^*(\tilde\tau_j^2, z)\|_\alpha = \Gamma_9 \|y - z\|_S,
\label{rizn11}
\end{eqnarray}
where $\Gamma_9 < 1$ uniformly for $j$ and $y, z \in \mathcal{N}_{\varrho}.$

By (\ref{rizn1}), (\ref{rizn9}) and (\ref{rizn11}) we conclude that the map $S: \ \mathcal{N}_{\varrho} \to \mathcal{N}_{\varrho}$ is a contraction.
Therefore, there exists unique almost periodic sequence $y^* = \{y_j^*\}$ such that
$u^*(\tau_j( y_j^*), y^*) =  y_j^*$
for all $j \in {Z}.$ The function $u^*(t, y^*)$ is
 W-almost periodic solution of the equation (\ref{ban1w}), (\ref{ban2w}).

\vspace{2mm}

2. Now we  proof the stability of the almost periodic solution.
Fix arbitrary $\varepsilon > 0$ and $\eta > 0.$ Let $t_0 \in [\tau_0(0) + \eta, \tau_1(0) - \eta].$

 The W-almost periodic solution $u_0(t)$ satisfies the integral equation
 \begin{eqnarray*} 
  u_0(t) = U_0(t,t_0)u_0 + \int_{t_0}^t U_0(t,s)f(s, u_0(s)) ds + \sum_{t_0 < \tau_j^0 < t} U_0(t, \tau_j^0+0) g_j(\tau_j^0),
 \end{eqnarray*}
where $\tau_j^0 = \tau_j(u_0(\tau_j^0))$ and $U_0(t,s)$ is the evolution operator of the linear equation
 \begin{eqnarray*}
 \frac{du}{dt} + Au = 0, \quad u(\tau_j^0 + 0) - u(\tau_j^0) = B_j u(\tau_j^0), \ j = 1,2,....
  \end{eqnarray*}
Let $u_1 \in X^\alpha$ such that $\| u_0 - u_1\|_\alpha < \delta.$ The solution $u_1(t)$ with initial value $u_1(t_0) = u_1$ satisfies equation
 \begin{eqnarray*} 
  u_1(t) = U_1(t,t_0)u_1 + \int_{t_0}^t U_1(t,s)f(s, u_1(s)) ds + \sum_{t_0 < \tau_j^1 < t} U_1(t, \tau_j^1+0) g_j(\tau_j^1),
 \end{eqnarray*}
where $\tau_j^1 = \tau_j(u_1(\tau_j^1))$ and $U_1(t,s)$ is the evolution operator of the linear equation
 \begin{eqnarray*}
 \frac{du}{dt} + Au = 0, \quad u(\tau_j^1 + 0) - u(\tau_j^1) = B_j u(\tau_j^1), \ j = 1,2,....
  \end{eqnarray*}

 By Lemma \ref{lem4}, for a sufficiently small Lipschitz constant $N_1$ the evolution operator $U_0(t,s)$
 satisfies the inequality
  \begin{eqnarray} \label{1st}
 \|U_0(t,s)u\|_\alpha \le M_1 e^{-\beta_1(t - s)}\|u\|_\alpha, \ t \ge s,
  \end{eqnarray}
 with some positive constants $\beta_1 \le \beta, \ M_1 \ge M.$
 Moreover, one can verify that for some domain $U^\alpha_{\tilde\rho}, \tilde\rho \le \rho,$
 and $N_1 \le N_0$ the evolution operator $U_1$ satisfies
   \begin{eqnarray} \label{2st}
 \|U_1(t,s)u\|_\alpha \le M_1 e^{-\beta_1(t - s)}\|u\|_\alpha, \ t \ge s, \ t,s \in [t_0, t_0 + T],
  \end{eqnarray}
 if the values $u_1(t)$ belong to $U^\alpha_{\tilde\rho}$ for $t \in [t_0, t_0 + T]$.

At the interval without impulses, the difference of solutions $u_0(t) - u_1(t)$ satisfies the inequality
 \begin{eqnarray*}
 & &  \|u_1(t) - u_0(t) \|_\alpha \le \|e^{-A(t-t_1)}(u_0(t_1) - u_1(t_1))\|_\alpha  \\
 & & + \int_{t_1}^t \|A^\alpha e^{-A(t-t_1)}(f(s,u_1(s)) - f(s,u_0(s)))\|ds \\
 & & \le M_1 e^{-\beta_1(t - t_1)}\|u_0(t_1) - u_1(t_1)\|_\alpha + \int_{t_1}^t \frac{M_1 N_1 e^{-\beta_1(t - s)}}{(t-s)^\alpha}\|u_1(s)-u_0(s)\|_\alpha ds.
 \end{eqnarray*}
Then by Lemma \ref{lem-gr},
\begin{eqnarray} \label{st22}
 \|u_1(t) - u_0(t) \|_\alpha \le M_1 \tilde C  e^{-\beta_1(t - t_1)} \|u_1(t_1) - u_0(t_1)\|_\alpha, \ t - t_1 \le Q.
 \end{eqnarray}
Hence, if initial values belong to the bounded domain from $X^\alpha,$ then the corresponding solutions
are uniformly bounded for $t$ from the bounded interval.

Assume for definiteness that $\tau_j^0 \ge \tau_j^1$ and estimate $|\tau_j^1 - \tau_j^0|$ by  $(u_1(\tau_j^1) - u_0(\tau_j^1)).$ By (\ref{rizn10})
 \begin{eqnarray*}
 & & \|(u_1(\tau_j^1) - u_0(\tau_j^0)\|_\alpha \le \|(u_0(\tau_j^1) - u_0(\tau_j^0)\|_\alpha + \|u_0(\tau_j^1) - u_1(\tau_j^1)\|_\alpha \\
 & & \le \Bigl\|\int_{\tau_j^1}^{\tau_j^0} \frac{d}{d\xi}u_0(\xi)d\xi\Bigl\|_\alpha + \|u_0(\tau_j^1) - u_1(\tau_j^1)\|_\alpha \\
  & & \le \tilde K_2 |\tau_j^0 - \tau_j^1| + \|u_0(\tau_j^1) - u_1(\tau_j^1)\|_\alpha.
 \end{eqnarray*}
 Hence,
  \begin{eqnarray}
  |\tau_j^0 - \tau_j^1| \le N_1 \|u_0(\tau_j^0) - u_1(\tau_j^1)\|_\alpha \le
\frac{N_1}{1 - \tilde K_2 N_1} \|u_0(\tau_j^1) - u_1(\tau_j^1)\|_\alpha.  \label{st5}
 \end{eqnarray}

Denote $\tau'_j = \min\{\tau_j^0, \tau_j^1\}, \ \tau''_j = \max\{\tau_j^0, \tau_j^1\}, \ j = 1,2,....$
We assume that $t \in (\tau''_i,\tau'_{i+1}]$ and estimate the difference
 \begin{eqnarray}
 & &  \|u_0(t) -  u_1(t)\|_\alpha = \| U_0(t, t_0)(u_0 - u_1)\|_\alpha +
 \|( U_0(t, t_0) -  U_1(t, t_0))u_1\|_\alpha \nonumber\\
 & & + \int_{t_0}^t \|U_0(t,s)f(s, u_0(s)) - U_1(t,s)f(s, u_1(s))\|_\alpha ds  \nonumber\\
& & + \|\sum_{t_0 < \tau_j^1 < t} U(t, \tau_j^1+0) g_j(\tau_j^1)
 - \sum_{t_0 < \tau_j^0 < t} U(t, \tau_j^0+0) g_j(\tau_j^0)\|_\alpha \nonumber\\
 & &  \le \| U_0(t, t_0)(u_0 - u_1)\|_\alpha + \|( U_0(t, t_0) -  U_1(t, t_0))u_1\|_\alpha \nonumber\\
  & & + \int_{t_0}^{\tau'_1} \|U_0(t,s)f(s, u_0(s)) - U_1(t,s)f(s, u_1(s))\|_\alpha ds + \nonumber\\
 & & + \sum_{j=1}^{i-1} \int_{\tau''_j}^{\tau'_{j+1}} \|(U_0(t,s) - U_1(t,s))f(s, u_1(s))\|_\alpha ds \nonumber\\
 & & + \sum_{j=1}^{i-1} \int_{\tau''_j}^{\tau'_{j+1}} \|U_0(t,s)(f(s, u_0(s)) - f(s, u_1(s)))\|_\alpha ds  \nonumber\\
& & + \sum_{j=1}^{i} \int_{\tau'_j}^{\tau''_j} \|U_0(t,s)f(s, u_0(s)) - U_1(t,s)f(s, u_1(s))\|_\alpha ds \nonumber\\
& & + \int_{\tau''_i}^t \|U_0(t,s)f(s, u_0(s)) - U_1(t,s)f(s, u_1(s))\|_\alpha ds \nonumber\\
& & + \sum_{j=1}^{i} \| U_0(t, \tau_j^0+0) g_j(\tau_j^0) -  U_1(t, \tau_j^1+0) g_j(\tau_j^1)\|_\alpha. \label{st6}
 \end{eqnarray}

Denote $v(t) =  \|u_0(t) -  u_1(t)\|_\alpha.$ Assume that for $t \in [t_0, \tau'_i]$ the values $u(t)$ belong to
$U^\alpha_{\tilde\rho};$ hence, the evolution operators $U_0(t,\tau)$ and $U_1(t,\tau)$ satisfy (\ref{1st}) and (\ref{2st})
at this interval.
By (\ref{st6}), analogous to the proof of (\ref{rizn1aa}), (\ref{in5}), and (\ref{rizn2}), we conclude that there exist positive constants $M_2$
and $P_1$ independent of $i$ such that
 for $t \in \mathcal{J}_{i+1}$
 \begin{eqnarray}
  & &  v(t) \le M_1 e^{-\beta_1 (t - t_0)} v(t_0) +
  \int_{t_0}^{\tau'_1} M_2 N_1 e^{-\beta_1 (t - \tau''_1)}  v(s) ds  \nonumber\\
  & & + \sum_{j=2}^{i-1}  \int_{\tau''_{j-1}}^{\tau'_{j}} M_2 N_1 e^{-\beta_1 (t - \tau''_{j})} v(s) ds +
   \sum_{j=1}^{i -1} P_1 N_1 e^{-\beta_1 (t - \tau''_j)}  v(\tau'_j)  \nonumber\\
   & & + \frac{1}{(t - \tau''_i)^{\alpha_1}}\left(\int_{\tau''_{i-1}}^{\tau'_{i}} M_2 N_1 e^{-\beta_1 (t - \tau''_{i})} v(s) ds +
  P_1 N_1 e^{-\beta_1 (t - \tau''_i)}  v(\tau'_i)\right)  \nonumber \\
   & & + \int_{\tau''_i}^t M_2 N_1 e^{-\beta_1 (t - s)}(t - s)^{-\alpha_1} v(s) ds, \quad \alpha_1 > \alpha.
  \label{st66}
 \end{eqnarray}
 Denote $\tilde Q = \max_j\{1, (\tau'_{j+1} - \tau''_j)\}, \ \tilde \theta = \min_j\{1, (\tau'_{j+1} - \tau''_j)\},
 \ P_2 = P_1 e^{\beta_1 \sup_j|\tau''_j - \tau'_j|},$ \ $ M_3 = M_2 e^{\beta_1 Q}.$
 By (\ref{st22}), at the interval $[t_0, \tau'_1]$ the function
$v(t)$ satisfies
\begin{eqnarray} \label{st9}
 v(t) \le  M_1 \tilde C e^{-\beta_1 (t - t_0)} v(t_0), \  t \in [t_0, \tau'_1].
  \end{eqnarray}
By (\ref{st66}) and (\ref{st9}), for $t \in (\tau''_1, \tau'_2]$ we get
 \begin{eqnarray*}
 & &  v(t) \le M_1 e^{-\beta_1 (t - t_0)} v(t_0) + \frac{1}{(t - \tau''_1)^{\alpha_1}}\int_{t_0}^{\tau'_1} M_2 N_1 e^{-\beta_1 (t - \tau''_1)} v(s) ds  \\
  & &  +  P_1 N_1 e^{-\beta_1 (t - \tau''_1)} (t - \tau''_1)^{-{\alpha_1}} v(\tau'_1) +
  \int_{\tau''_1}^t M_2 N_1 e^{-\beta_1 (t - s)}(t - s)^{-\alpha_1} v(s) ds.
 \end{eqnarray*}
 Hence, for $v_1(t) = e^{\beta_1 t} v(t)$
  \begin{eqnarray*}
   v_1(t) \le M_1 v_1(t_0) \left( 1 + \frac{N_1 \tilde C (M_3 \tilde Q + P_2)}{(t - \tau''_1)^{\alpha_1}}\right) +
    \int_{\tau''_1}^t M_2 N_1(t - s)^{-\alpha_1} v_1(s) ds.
  \end{eqnarray*}
By Lemma \ref{lem-gr}
\begin{eqnarray*}  
 v(t) \le M_1 \tilde C_1  v(t_0) e^{-\beta_1 (t - t_0)}\left( 1 + \frac{N_1 \tilde C (M_3 \tilde Q + P_2)}{(t - \tau''_1)^{\alpha_1}}\right),
 \ t \in (\tau''_1, \tau'_2],
\end{eqnarray*}
where $\tilde C_1$ is defined by (\ref{evop222}).

 Let us prove that \ \ $ v(t) \le$
 \begin{eqnarray} \label{st11}
\le M_1 \tilde C_1  v(t_0) e^{-\beta_1 (t - t_0)} \left( 1 + \frac{N_1 \tilde C_1 (M_2\tilde Q + P_2)}{(t - \tau''_j)^{\alpha_1}}\right)
  \left( 1 + \frac{N_1 \tilde C_1 (M_2\tilde Q + P_2)}{(1 - \alpha_1)\tilde \theta^{\alpha_1}}\right)^{i-1}
 \end{eqnarray}
 for $t \in (\tau''_i, \tau'_{i+1}], \ i \ge 2.$ We apply the method of mathematical induction.
 Assume that (\ref{st11}) is true for $t \in (\tau''_{n-1}, \tau'_{n}]$ and prove it for $t \in (\tau''_{n}, \tau'_{n+1}].$ Really, by (\ref{st66})
 for $t \in (\tau''_{n}, \tau'_{n+1}]$ we have
 \begin{eqnarray*}
 & &  v(t) \le M_1  e^{-\beta_1 (t - t_0)} v(t_0) \Biggl( \left( 1 + (M_3 \tilde Q + P_2)N_1 \tilde C\right)  \nonumber \\
 & & + \sum_{j=2}^{n-1} {\mathcal A}^{j} M_3 N_1 \tilde Q \tilde C_1 +  \sum_{j=2}^{n-1} {\mathcal A}^{j-1}
 \left( 1 +  \frac{N_1 \tilde C_1(M_3\tilde Q + P_2)}{\tilde\theta^{\alpha_1}}\right) N_1 P_2 \tilde C_1  \nonumber \\
  & & + {\mathcal A}^{n-2} \Biggl(N_1 M_2 \tilde C_1\Bigl(\tau'_n -\tau''_{n-1}) +
  \frac{N_1 \tilde C_1(M_3\tilde Q + P_2)(\tau'_n -\tau''_{n-1})}{(1 - \alpha_1)(\tau'_n -\tau''_{n-1})^{\alpha_1}}\Bigl)  \nonumber \\
 & & +  N_1 P_2 \tilde C_1 \Bigl( 1 + \frac{N_1 \tilde C_1 (M_3\tilde Q + P_2)}{(\tau'_n -\tau''_{n-1})^{\alpha_1}}\Bigl)\Biggl) +
{\mathcal B_n(t)} \nonumber \\
    & & \le {\mathcal A} + \sum_{j=2}^{n-1}{\mathcal A}^{j-1}(1 + N_1 \tilde C_1 (M_3 \tilde Q + P_2) - 1)+
    \frac{{\mathcal A}^{n-1} N_1 \tilde C_1 (M_3 \tilde Q + P_2)}{(t - \tau''_n)^{\alpha_1}}  \nonumber \\
    & & +  {\mathcal B_n(t)} \le
    {\mathcal A}^{n-1}\left( 1 + \frac{N_1 \tilde C_1 (M_3 \tilde Q + P_2)}{(t - \tau''_n)^{\alpha_1}}\right) + {\mathcal B_n(t)},
 \end{eqnarray*}
 where
 $${\mathcal A} = \left( 1 + \frac{N_1 \tilde C_1 (M_3\tilde Q + P_2)}{(1 - \alpha_1)\tilde \theta^{\alpha_1}}\right), \
 {\mathcal B_n(t)} =
 \int_{\tau''_n}^t \frac{M_3 N_1}{(t - s)^{\alpha_1}} e^{-\beta_1 (t - s)} v(s) ds.$$

  Hence, for $t \in (t_n, t_{n+1}]
  ,$ the function $v_1(t) = e^{\beta_1 t} v(t)$ satisfies the inequality
 $$v_1(t) \le {\mathcal A}^{n-1}\left( 1 +\frac{ N_1 \tilde C_1 (M_3 \tilde Q + P_2)}{(t - \tau''_n)^{\alpha_1}}\right) +
  M_3 N_1 \int_{\tau''_n}^t (t - s)^{-\alpha_1}v_1(s)ds.$$
Applying Lemma \ref{lem-gr}, we obtain (\ref{st11}).

Let $N_1 > 0$ be such that ${\mathcal A}^{{i}(t_0,t)} e^{-\beta_1 (t-t_0)} < e^{-\delta_1(t - t_0)}$
for some positive $\delta_1.$
For given $\varepsilon > 0$ and $\eta > 0$ we choose $v(t_0) = v_0$ such that
$$M_1 \tilde C_1 v_0 \left( 1 +\frac{ N_1 \tilde C_1 (M_3 \tilde Q + P_2)}{\eta^{\alpha_1}}\right) < \varepsilon.$$
This proves the asymptotic stability of solution $u_0(t).$
\end{proof}

{\bf Example 1.}
Let us consider the parabolic equation with impulses in variable moments of time
\begin{eqnarray}
& & u_t = u_{xx} + a(t)u_x + b(t,x), \label{ex1} \\
& & \Delta u\Biggl|_{t = \tau_j(u)} = u(\tau_j(u)+0,x) -  u(\tau_j(u),x) =  -a_j  u(\tau_j(u),x), \label{ex2}
\end{eqnarray}
with boundary conditions
\begin{eqnarray} \label{ex3}
u(t,0) = u(t,\pi) = 0,
\end{eqnarray}
where the sequence of surfaces $\tau_j(u)$ is defined by
\begin{eqnarray*}  
\tau_j(u) = \theta_j +  b_j \int_0^\pi u^2(\xi)d\xi, \ j \in Z,
\end{eqnarray*}
where the sequence of real numbers $\{\theta_j\}$ has
uniformly almost periodic sequences of differences and $\theta_{j+1} - \theta_j \ge \theta \ge 1/2,$

$\{a_j\}$ and $\{b_j\}$ are almost periodic sequences of positive numbers,

$a(t)$ is a Bohr almost periodic function,

$b(t,x)$ is a Bohr almost periodic function in $t$ uniformly with respect to $x \in [0,\pi]$ and belongs to $L_2(0,\pi)$
for all fixed $t.$

Denote $$X = L_2(0,\pi), \ A = - \frac{\partial^2}{\partial x^2}, \ X^1 = D(A) = H^2(0,\pi)\cap H^1_0(0,\pi).$$

The operator $A$ is sectorial with simple eigenvalues $\lambda_k = k^2$ and corresponding eigenfunctions
$$\varphi_k(x) = \left(\frac{2}{\pi}\right)^{1/2}\sin kx, \ k = 1,2,....$$
The operator $(-A)$ generates an analytic semigroup $e^{-At}.$
Let $u = \sum_{k=1}^\infty a_k \sin k x, a_k = \frac{1}{\pi}\int_0^\pi u(x)\sin kx dx.$
Then
$$A u = \sum_{k=1}^\infty k^2 a_k \sin k x, \  A^{\alpha} u = \sum_{k=1}^\infty k^{2\alpha} a_k \sin k x, \
e^{-At} = \sum_{k=1}^{\infty}e^{-k^2 t}a_k \sin kx.$$
Hence,
$$X^{1/2} = D(A^{1/2}) = H^1_0(0,\pi).$$

Let us consider Eq. (\ref{ex1}) - (\ref{ex3}) in space $X^{1/2} = D(A^{1/2}) = H^1_0(0,\pi):$
$$\frac{du}{dt} + Au = f(t,u), \quad u(\tau_j(u) + 0) = (1 - a_j)u(\tau_j(u)), \ j \in Z,$$
where $f(t,u): R \times X^{1/2} \to X, \ f(t,u)(x) = a(t)u_x + b(t,x).$

We verify that in some domain $\mathcal{D} = \{u \ge 0, \|u\| \le \rho\}$ solutions of (\ref{ex1}) –- (\ref{ex3})
don't have beating at the surfaces $t = \tau_j(u).$
Assume to the contrary that solution $u(t)$ intersects
 the surface $t = \tau_j(u)$ at two points $t_j^1$
 and $t_j^2$, $t_j^1 < t_j^2.$

Denote $u(t_j^1) = u_1, u(t_j^2) = u_2, \tilde u = e^{-A(t_j^2 - t_j^1)}u(t_j^1+0).$
 Then $u(t_j^1+0) = (1 - a_j)u_1,$ $\tau_j(u_1)) = t_j^1,$ $ \tau_j(u_2)) = t_j^2,$
 and
$$u_2 = e^{-A(t_j^2 - t_j^1)}u(t_j^1+0) + \int_{t_j^1}^{t_j^2}  e^{-A(t_j^2 - s)}
f(s,u(s))ds.$$
We have
\begin{eqnarray*}
& & |\tau_j(u_2) -  \tau_j(\tilde u)| \le
 b_j \int_0^\pi |(u_2(t,x) - \tilde u(t,x))(u_2(t,x) + \tilde u(t,x))|dx  \nonumber\\
& & \le b_j \|u_2(t,x) - \tilde u(t,x)\|_{L_2} \|u_2(t,x) + \tilde u(t,x)\|_{L_2}  \nonumber\\
& & \le  b_j \int_{t_j^1}^{t_j^2} \| e^{-A(t_j^2 - s)}
f(s,u(s))\|_{L_2}ds \|u_2(t,x) + \tilde u(t,x)\|_{L_2}.
\end{eqnarray*}
The function $f(t,u)$ satisfies $\| f(t,u)\|_{X} \le K(1 + \|u\|_{X^{1/2}});$
hence, solutions of the equation without impulses exist for all $t \ge t_0$ and there exist
positive constants $M_1$ and $M_2$ such that
 $M_2 \ge \sup_{u \in \mathcal{D}} \|f(t, u)\|_{L_2},$ $M_3 \ge \sup_{u \in \mathcal{D}} \|u_2(t,x) + \tilde u(t,x)\|_{L_2}.$
 Therefore, $|\tau_j(u_2) -  \tau_j(\tilde u)| \le  b_j |t_j^2 - t_j^1| M_2 M_3.$
 By sufficiently small $b = \sup_j b_j$ we have $b M_2 M_3 < 1$ and
\begin{eqnarray*}
& & 0 < t_j^2 - t_j^1 = \tau_j(u_2) -  \tau_j(u_1) \le
\tau_j(u_2) -  \tau_j(\tilde u) + \tau_j(\tilde u) -  \tau_j(u_1), \\
 & &  t_j^2 - t_j^1 \le \frac{1}{1 - b M_2 M_3}(\tau_j(\tilde u) - \tau_j(u_1)) \le
 \frac{b_j((1-a_j)^2 -1)}{1 - b M_2 M_3}\|u_1\|^2_{L_2} < 0.
\end{eqnarray*}
This contradicts our assumption.

Corresponding to (\ref{ex1}) - (\ref{ex3}), the linear impulsive equation is exponentially stable in space $X^{1/2}.$
By Theorem \ref{thm4}, for sufficiently small $b = \sup_j b_j$ and $a = \sup_t |a(t)|$ the equation has
an asymptotically stable W-almost periodic solution.


\section{Equations with unbounded operators $B_j.$}
Many results in our paper remain true if operators $B_j$ in linear parts of impulsive action are unbounded.
We refer to \cite{T3}, where the following semilinear impulsive differential equation
\begin{eqnarray} \label{ban1un}
& & \frac{du}{dt} = A u + f(t, u), \quad t \not= \tau_j, \\
& & \Delta u|_{t=\tau_j} = u(\tau_j) - u(\tau_j-0) = B_j u(\tau_j-0) + g_j(u(t_j-0)), \quad j \in  Z, \label{ban2un}
\end{eqnarray}
was studied. Here $u: {R} \to X,$ $X$ is a Banach space, $A$ is a sectorial operator in $X$,
$\{B_j\}$ is a sequence of some closed operators, and $\{\tau_j\}$ is an unbounded and strictly increasing sequence of real numbers.
Assume that the equation satisfies conditions $(H1), (H3), (H5), (H6)$ and

${\bf (H4u)}$  the sequence $\{B_j\}$ of closed linear operators $B_j \in L(X^{\alpha+\gamma}, X^\alpha)$ is almost periodic
in the space $L(X^{\alpha+\gamma}, X^\alpha),$ for $\alpha \ge 0$ and some $\gamma > 0.$
\vspace{1mm}

As in \cite{RT}, we assume that solutions $u(t)$ of (\ref{ban1un}), (\ref{ban2un}) are right-hand-side continuous;
hence $u(\tau_j) = u(\tau_j + 0)$ at all points of impulsive action.
Due to such a selection we avoid considering operators $e^{-A(t-\tau_j)}(I + B_j)$ with unbounded operators $B_j$
and can work with the family of bounded operators $e^{-A(t-\tau_j)}.$

Since the operator $A$ is sectorial and operators $B_j$ are subordinate to $A$, an evolution operator of a corresponding linear
impulsive equation is constructed correctly.
Now analogs of the theorems \ref{theorem2} and \ref{theorem3} can be proved.
\vspace{2mm}

{\bf Example 2.} \cite{T3}.
We consider the following parabolic equation with impulsive action:
\begin{eqnarray}
& & u_t = u_{xx} + f(t,x), \label{2ex1} \\
& & \Delta u\Bigl|_{t = \tau_j} = u(\tau_j,x) -  u(\tau_j - 0,x) = b_k (\sin x )u_x + c_k x(\pi - x), \label{2ex2}
\end{eqnarray}
with boundary conditions
\begin{eqnarray} \label{2ex3}
u(t,0) = u(t,\pi) = 0,
\end{eqnarray}
where
$\{\tau_j\}$ is a sequence of real numbers with uniformly almost periodic sequences of differences, $\tau_{j+1} - \tau_j \ge \theta \ge 1/2,$

$\{b_j\}$ and $\{c_j\}$ are almost periodic sequences of real numbers,

$f(t,x)$ is almost periodic and locally Holder continuous with respect to $t$ and for every fixed $t$ belongs to $L_2(0,\pi).$

As in Example 1, denote $$X = L_2(0,\pi), \ A = - \frac{\partial^2}{\partial x^2}, \ X^1 = D(A) = H^2(0,\pi)\cap H^1_0(0,\pi).$$

The operator $A$ is sectorial with simple eigenvalues $\lambda_k = k^2$ and corresponding eigenfunctions $\varphi_k(x) = \sin kx,
k = 1,2,....$

If $u = \sum_{k=1}^\infty a_k \sin k x, \
a_k = \frac{1}{\pi}\int_0^\pi u(x)\sin kx dx,$ then
\begin{eqnarray*}
B_j u = b_j \sin x u_x = b_j \sin x \sum_{k=1}^\infty a_k k \cos k x = \frac{b_j}{2}(R-L)A^{1/2}u = b_j T A^{1/2}u,
\end{eqnarray*}
where $Ru = \sum_{k=1}^\infty a_k \sin (k-1) x$ and  $Lu = \sum_{k=1}^\infty a_k \sin (k+1) x$ are bounded shift operators
in $X.$ Hence, operators $B_j: X^{\alpha+1/2} \to X^{\alpha}$ are linear continuous, $\alpha \ge 0.$

Analogous to (\ref{evol1}), the evolution operator for homogeneous equation (\ref{2ex1}), (\ref{2ex2}) is
$$U(t,s) = e^{-A(t-s)}, \ \ {\rm if} \ \ \tau_{k} \le s \le t < \tau_{k+1},$$
and
$$U(t,s) = e^{-A(t-\tau_k)}(I + B_k)e^{-A(\tau_k - \tau_{k-1})}...(I + B_m)e^{-A(\tau_m - s)}$$
if $\tau_{m-1} \le s < \tau_m < \tau_{m+1}... \tau_k \le t < \tau_{k+1}, \ m < k, \ k, m \in {Z}.$
\vspace{2mm}

\begin{theorem}
Let $p\ln(1+b) < 1,$ where $p$ is defined by (\ref{limp}) and $b = \sup_j|b_j|.$
Then equation (\ref{2ex1}), (\ref{2ex2}) with boundary conditions (\ref{2ex3}) has a unique W-almost periodic solution
which is asymptotically stable.
\end{theorem}

\begin{proof}
We show that the unique almost periodic solution of (\ref{2ex1}) and (\ref{2ex2}) is given as function
${R} \to L_2(0,\pi)$ by formula
 \begin{eqnarray*}
u_0(t) = \int_{-\infty}^{t} U(t,s)\tilde f(s) ds + \sum_{\tau_j \le t} U(t, \tau_j) \tilde g_j,
\end{eqnarray*}
where $\tilde f(t) \equiv f(t,.): {R} \to L_2(0,\pi), g_j(x) = c_j x(\pi - x), \tilde g_j = g_j(.): {Z} \to L_2(0,\pi).$

First, $u_0(t)$ is bounded in space $X^\alpha.$ If $t \in [\tau_i, \tau_{i+1})$ then
 \begin{eqnarray}
& &  \int_{-\infty}^{t} \|U(t,s)\tilde f(s)\|_\alpha  ds  \le
\int_{\tau_{i-1}}^{\tau_i} \|A^\alpha e^{-A(t-\tau_i)}(I + B_i)e^{-A(\tau_i-s)}\tilde f(s)\|ds  \nonumber\\
 & & + \int_{\tau_i}^t \| A^\alpha e^{-A(t-s)}\tilde f(s)\|ds + \sum_{k=2}^\infty \int_{\tau_{i-k}}^{\tau_{i-k+1}}
 \|A^\alpha e^{-A(t-\tau_i)}(I + B_i)e^{-A(\tau_i - \tau_{i-1})}\|  \nonumber\\
 & & \times \prod_{j=i-1}^{i-k+2}\|(I + B_{j})e^{-A(\tau_{j} - \tau_{j-1})}\|
 \|(I + B_{i-k+1})e^{-A(\tau_{i-k+1} - s)}\tilde f(s)\|ds,     \label{ex11}
\end{eqnarray}

Next, we need the following inequality (see \cite{RT}), p. 35):
 \begin{eqnarray} \label{ex12}
\|A^\alpha T A^\beta e^{-At}\| = \frac12\|A^\alpha (R-L)A^\beta e^{-At}\|
\le \frac{4^\alpha +1}{2}\|A^{\alpha +\beta} e^{-At}\|.
\end{eqnarray}
Then by (\ref{ex12}):
 \begin{eqnarray}
& &  \|A^\alpha e^{-A(t-\tau_i)}(I + B_i)e^{-A(\tau_i-s)}\| \le \|A^\alpha e^{-A(t-s)}\| + \frac{5}{2} \|A^{\alpha +1/2} e^{-A(t-s)}\|  \nonumber\\
& &  \le \left( C_\alpha (t - s)^{-\alpha} + \frac{5}{2}C_{\alpha + 1/2} (t - s)^{-(\alpha+1/2)}\right) e^{-\delta(t-s)}. \label{ex13}
\end{eqnarray}

From Henry \cite{H}, p. 25, we have
$$\|A^\alpha e^{-At}\psi\| < b_\alpha(t)\|\psi\|,$$
where $b_\alpha(t) = (te/\alpha)^{-\alpha}$ if $0 < t \le \alpha/\lambda_1,$
and $b_\alpha(t) = \lambda_1^\alpha e^{-\lambda_1 t}$ if $t \ge \alpha/\lambda_1.$
Since $\|T\| = 1,$ and $\lambda_1 =1,$ we have
 \begin{eqnarray*}
& & \|(I+B_j)e^{-A(\tau_j - \tau_{j-1})}\| \le \|e^{-A(\tau_j - \tau_{j-1})}\| +
|b_j|\|A^{1/2}e^{-A(\tau_j - \tau_{j-1})}\|  \nonumber\\
& & \le (1+|b_j|)e^{-(\tau_j - \tau_{j-1})}. 
\end{eqnarray*}

Let $0 <\varepsilon_1 < 1 - p\ln(1+b).$ Then there exists a positive integer $k_1$ such that for $k \ge k_1$
$$\frac{{i}(\tau_{i-k}, \tau_{i})}{\tau_{i} - \tau_{i-k}} \ln(1+b) - 1 < -\varepsilon_1.$$
Denote
$N_1 = \max_{1 \le k \le k_1}\exp\left({i}(\tau_{i-k}, \tau_{i})\ln(1+b) - (\tau_i - \tau_{i-k})\right).$
Then
\begin{eqnarray}
& & \prod_{j=i-k+1}^{i}\|(I+B_{j})e^{-A(\tau_j - \tau_{j-1})}\| \le (1+b)^{{i}(\tau_{i-k}, \tau_{i})} e^{-(\tau_{i} -
\tau_{i-k})}  \nonumber\\
& & \le N_1 e^{-\varepsilon_1(\tau_{i} - \tau_{i-k})} \le N_1 e^{-\varepsilon_1\theta k}. \label{ex15}
\end{eqnarray}
For $t \in (\tau_i, \tau_{i+1}),$ by (\ref{ex13}) and (\ref{ex15}) we get
\begin{eqnarray*}
& & \| U(t, \tau_{i-k})\|_\alpha \le \|A^\alpha e^{-A(t-\tau_i)}(I + B_i)e^{-A(\tau_i - \tau_{i-1})}\| \times \\
 & & \times \prod_{j=i-k+1}^{i-1}\|(I+B_{j})e^{-A(\tau_j - \tau_{j-1})}\| \le K_1 e^{-\varepsilon_1(t - \tau_{i-k})}
\end{eqnarray*}
with constant $K_1$ independent of $t$ and $\tau_{i-k}.$

Using the last inequality, we obtain the boundedness of
$\|u_0(t)\|_\alpha$.
We can now proceed analogously to the proof of Theorem 1 and show the almost periodicity of $u_0(t).$
\end{proof}

\end{document}